\documentclass{article}
\usepackage{amsmath,amssymb,amsthm,graphicx,epsfig,float,url}
\usepackage[colorlinks=true,citecolor={Plum},linkcolor={Periwinkle}]{hyperref}
\usepackage{pdfsync}
\usepackage[usenames, dvipsnames]{xcolor}
\usepackage{tikz}
\usepackage{subfig}
\usepackage{bbm}
\usepackage{stmaryrd}
\usepackage{amssymb}
\usepackage{mathrsfs}  
\usepackage{caption}
\usepackage{float}
\usepackage{lineno}
\usepackage[utf8]{inputenc}
\usepackage{dsfont}
\usepackage[makeroom]{cancel}
\usetikzlibrary{patterns}
\newcommand{\R}{\textnormal{I\kern-0.21emR}}
\newcommand{\N}{\textnormal{I\kern-0.21emN}}

\renewcommand{\geq}{\geqslant}
\renewcommand{\leq}{\leqslant}

\def\e{{\varepsilon}}

\def\om{{\overline{m}}}

\allowdisplaybreaks

\def\YYint#1#2#3{{\setbox0=\hbox{$#1{#2#3}{\iint}$}
    \vcenter{\hbox{$#2#3$}}\kern-.51\wd0}}
 
\newcommand{\corr}[1]{{{\color{black}#1}}}

\usepackage[dvipsnames]{xcolor}

\newtheorem*{theorem*}{Theorem}

\newtheorem{theorem}{Theorem}

\newtheorem{material}{material}
\newtheorem{proposition}[material]{Proposition}
\newtheorem{corollary}[material]{Corollary}
\newtheorem{definition}[material]{Definition}
\newtheorem{lemma}[material]{Lemma}

\newtheorem{example}[material]{Example}
\newtheorem{remark}[material]{Remark}

\def\O{{\Omega}}
\def\n{{\nabla}}
\def\p{{\varphi}}
\def\TT{{(0,T)\times \O}}
\def\TeT{{(0,T-\varepsilon)\times \T}}

\def\T{{\mathbb T}}
\def\TT{{(0,T)\times \T}}
\def\OT{{(0,T)\times \O}}

\usepackage{xargs}
 \usepackage[colorinlistoftodos,textsize=small]{todonotes}
 \newcommandx{\unsure}[2][1=]{\todo[linecolor=red,backgroundcolor=red!25,bordercolor=red,#1]{#2}}
 \newcommandx{\change}[2][1=]{\todo[linecolor=blue,backgroundcolor=blue!25,bordercolor=blue,#1]{#2}}
 \newcommandx{\info}[2][1=]{\todo[linecolor=green,backgroundcolor=green!25,bordercolor=green,#1]{#2}}
 \newcommandx{\improvement}[2][1=]{\todo[linecolor=yellow,backgroundcolor=yellow!25,bordercolor=yellow,#1]{#2}}
 
  \newcommandx{\biblio}[2][1=]{\todo[linecolor=blue,backgroundcolor=magenta!25,bordercolor=blue,#1]{#2}}

 \numberwithin{equation}{section}

\begin{document}

\title{The bang-bang property in some parabolic  bilinear optimal control problems \emph{via}  two-scale asymptotic expansions }


\author{Idriss Mazari }
\date{\today}

\maketitle

\begin{abstract}
We investigate the bang-bang property for fairly general classes of $L^\infty-L^1$ constrained bilinear optimal control problems in {the case of parabolic equations set in the one-dimensional torus.}
 Such a study is motivated by several applications in applied mathematics, most importantly in the study of reaction-diffusion models. 
The main equation {writes} $\partial_t u_m-{\partial^2_{xx}} u_m=mu_m+f(t,x,u_m)$, where $m=m(x)$ is the control, which must satisfy some $L^\infty$ bounds ($0\leq m\leq 1$ a.e.) and an $L^1$ constraint ($\int m=m_0$ is fixed), and where $f$ is a non-linearity that must only satisfy that any solution of this equation is positive at any given time. The functionals we seek to optimise are rather general; they write $\mathcal J(m)=\iint_\TT j_1(t,x,u_m)+\int_\T j_2(x,u_m(T,\cdot))$. Roughly speaking we prove in this article that, if $j_1$ and $j_2$ are increasing, then any maximiser $m^*$ of $\mathcal J$ is bang-bang in the sense that it writes $m^*=\mathds 1_E$ for some subset $E$ of the torus. It should be noted that such a result rewrites as an existence property for a shape optimisation problem. Our proof relies on second order optimality conditions, combined with a fine study of two-scale asymptotic expansions. In the conclusion of this article, we offer several possible generalisations of our results to more involved situations (for instance for controls of the form $m\p(u_m)$ or for some time-dependent controls), and we discuss the limits of our methods by explaining which difficulties may arise in other settings.
\end{abstract}


\medskip

\paragraph{Acknowledgment.} This work was partially supported by the French ANR Project ANR-18-CE40-0013 - SHAPO on Shape Optimization and by the Project "Analysis and simulation of optimal shapes - application to life sciences" of the Paris City Hall.

The author warmly thanks G. Allaire, B. Geshkovski, G. Nadin and Y. Privat for scientific exchanges, {as well as the anonymous referee of the paper for his/her insightful comments.}

\paragraph{Contact Information:}CEREMADE, UMR CNRS 7534, Universit\'e Paris-Dauphine, Universit\'e PSL, Place du Mar\'echal De Lattre De Tassigny, 75775 Paris cedex 16, France.

\texttt{mazari@ceremade.dauphine.fr}

\newpage
\tableofcontents
\newpage
\section{Introduction}
\subsection{Scope of the paper, informal presentation of our results}

In this paper, we offer a theoretical analysis of a {class of optimal control problems}, in which one aims at optimising a criteria by acting in a bilinear way on the state of the PDE. Prototypically, the model under consideration reads as follows: for a given non-linearity $f=f(t,x,u)$ and a control $m=m(t,x)$, we let $u_m$ be the solution of 
$$\partial_t u_m-\Delta u_m=mu_m+f(t,x,u_m)$$  with variables $x\in \O$ and $t\in [0,T]$ under certain boundary conditions. For a certain time horizon $T>0$, we aim at optimising criteria of the form 
$$\mathcal J(m)=\iint_\OT j_1(t,x,u_m)+\int_\O j_2(x,u_m(T,\cdot))$$ under some constraints on the control $m$.  Throughout the paper, the constraints on $m$ will be of $L^1$ and $L^\infty$ type; in other words, one constraint takes the form {
\[ \forall t\in[0,T]\,, \int_\O m(t,x)dx=V_0\text{ fixed}\]  or \[\int_\O m(x)dx=V_0 \text{ (if $m$ does not depend on time)},\]}
while the other is 
\[{-\kappa_0}\leq m\leq {\kappa_1}\text{ a.e.}\]
In {this setting}, one of the salient qualitative features of optimisers is the \emph{bang-bang property}. In other words, is it true that any maximiser writes $m^*=\alpha+(\beta-\alpha)\mathds 1_E$ for some measurable subset $E$ of $\O$? This property is linked to (non-)existence results for shape optimisation problems.  There were, in recent years, several fine qualitative studies of this property in the elliptic case; we refer to Section \ref{Se:Bib}. However, in the context of parabolic models and despite the current activity in the study of parabolic bilinear optimal control problems, this property does not seem to be reachable by the available techniques; we refer to sections \ref{Se:Comment} and \ref{Se:Bib}. 

In {this paper}, we prove that, under reasonable assumptions on the non-linearity $f$ that ensure the well-posedness of the parabolic system, and on the cost functions $j_1\,, j_2$ (roughly speaking, they must both be non-decreasing, and one has to be increasing), the bang-bang property holds \emph{if we assume that admissible controls are  constant in time} and that \emph{the domain is one-dimensional}. This is the main contribution of this article. It hinges on the methods of \cite{MNPCPDE}, coupled with two-scale asymptotic techniques  previously used in \cite{MNT} in the context of the optimal control of initial conditions in reaction-diffusion equations. The reason why we tackle the one-dimensional periodic case will be explained later on. It should be noted that we explain in the conclusion how we may cover, with the same type of arguments, higher-dimensional orthotopes. The main explanation behind having to work with  time constant controls is a technical one; this allows to gain further regularity on the solutions of the parabolic PDE under consideration. For this reason, section \ref{Se:TDM} of the conclusion contains a discussion of possible generalisations and obstructions to generalisations; we explain, for instance, how to deal with controls $m$ writing $\sum_{i=1}^N \phi_i(t)m_i(x)$. As a first side comment, it should also be noted that our analysis covers the case of some \emph{tracking-type functionals}. This is not the main topic, and we refer to remark \ref{Re:Functional}. As a second side comment, our analysis can encompass more intricate interactions between the control and the state. For instance, we provide, in section \ref{Se:Interaction}  of the conclusion, a generalisation of our results to the case where the control and the state are coupled \emph{via} a term of the form $m\p(u_m)$ for a large class of $\p$.

\subsection{Main model and result}\label{Se:Model}
\subsubsection{The parabolic equation}
\paragraph{Admissible controls in parabolic models}
{We work} in the (one-dimensional) torus $\T$. In section \ref{ff}, we explain how our methods may extend to the case of higher dimensional tori. 

Regarding the time regularity of admissible resources distribution, we shall make a strong assumption: the admissible controls are  constant in time. The reason is that the method we introduce and develop hinges on fine regularity properties of solutions of the associated evolution equation that can not be obtained in the case where the control $m$ also depends on time. We refer to remark \ref{Re:TDDD} for further comments.

In this setting,  denoting by  $u_m$ the state of the equation and by $m$ the control, the only type of interaction we are interested in is bilinear; in other words, the control appears in the model \emph{via} the term $mu_m$ (see Remark \ref{Re:Interaction} and section \ref{Se:Interaction} for considerations on the case of interactions of the form $m\p(u_m)$). In terms of constraints, we impose two on the controls, an $L^\infty$ and an $L^1$ one.  Each of these constraints has a natural interpretation in different fields of applications. In spatial ecology for instance, one may think of $m$ as a resources distributions, in which case the $L^\infty$ constraint simply models the fact that, at any given point, there can only be a maximum amount of resources available, while the $L^1$ constraint accounts for the limitation of the global quantity of resources involved. For the $L^\infty$ constraints, without loss of generality (we also refer to remark \ref{Re:Constraints}), we shall consider controls satisfying 
\[0\leq m\leq 1\text{ a.e. in } \T.\]
For the $L^1$ constraint, we fix a volume constraint $V_0\in (0;{2\pi})$, and we shall consider controls satisfying 
\[\int_\T m=V_0.\]
This leads to considering the admissible class

\begin{equation}\label{Eq:Adm}\tag{$\bold{Adm}$}
\mathcal M(\T):=\left\{m\in L^\infty(\T):\, 0\leq m\leq 1\text{ a.e. in  }\T\,, \int_\T m=V_0\right\}.
\end{equation}
Of notable interest in $\mathcal M(\T)$ are \emph{bang-bang} functions; as they are the central theme of this paper we isolate their definition here.
\begin{definition}
A function $m\in \mathcal M(\T)$ is called \emph{bang-bang} if there exists $E\subset \T$ such that $m=\mathds 1_E$.
\end{definition} 
 
\paragraph{Nonlinearities under consideration}
Our choice of nonlinearity in the parabolic model also derives from considerations in mathematical biology or chemistry. Namely, we want the solutions not only to exist but to be uniformly bounded (in time) in the $L^\infty$ norm, as well as to enjoy a strong maximum property (in the sense that, starting from a non-zero initial condition, the solution is positive at any arbitrary positive time). The latter is not only important from a modelling point of view but also in the course of the proof, as it influences the monotonicity of the functional under consideration.

As the right hand side of the reaction-diffusion equation writes $mu+f(t,x,u)$ we shall make the following assumptions that guarantee the well-posedness of the ensuing system:
\begin{multline}\label{H1}\tag{$\bold H_1$} f \text{ is $\mathscr C^1$ in time, $L^\infty$ in $x$,  and $\mathscr C^2$ in $u$,}
\\\text{ and, for any $K\in \R$, $\sup_{x\in \T\,, u\in [0,K], t\in [0,K]}\left(\left\vert {\partial_t}f\right\vert+\left\vert {\partial_{u}}f\right\vert+\left\vert {\partial^2_{uu}}f\right\vert \right)<\infty.$}
\end{multline}
Assumption \eqref{H1} {serves to derive the regularity} of the solutions of the equation. The next assumption is used to obtain  upper and lower bounds on the solution:

\begin{multline}\label{H2}\tag{$\bold H_2$}
\text{ $f(\cdot,\cdot,0)\geq 0$, $f(\cdot,\cdot,0)\in L^\infty(\TT)$ and }
\\\text{there exists $\kappa>0$ such that for any  $u\geq \kappa$, for any $t\in \R_+$, for a.e. $x\in \T$,  $f(t,x,u)\leq -u$}.\end{multline}
In the first condition, if we had $f(\cdot,\cdot,0)=0$ this would simply model that when no individuals are present no reaction is happening. Assuming the general inequality allows to consider non-negative source terms (\emph{i.e.} one may take $f(t,x,u)=ug(t,x,u)+y(t,x)$ for a certain $g$ and a non-negative source term $y$). It should be noted that, had we taken $-\kappa_0\leq m\leq \kappa_1$ as $L^\infty$ constraints, the final inequality in \eqref{H2} would rewrite $f(t,x,u)\leq -\sup\{|\kappa_0|\,,|\kappa_1|\}u$. 

Finally, {we make a final assumption, which is seemingly the most restrictive one}, but we comment on it in Example \ref{Ex} and in Remark \ref{Cp}.
\begin{multline}\label{H3}\tag{$\bold H_3$}\text{ $f$ is, uniformly in $(t,x)\in \TT$, uniformly Lipschitz in $u\in \R$: there exists $A$ such that,}\\\text{ for any $(t,x)\in \TT\,, $ for any $u,u'\in \R$, $\left| f(t,x,u)-f(t,x,u')\right|\leq A|u-u'|.$}\end{multline}

\paragraph{Initial condition} We simply take an initial condition independent of $m$, say $u^0$, satisfying 
\begin{equation}\inf_\T u^0> 0\,,  u^0\in \mathscr C^2(\T).\end{equation} 
\paragraph{Parabolic model}
We define, for any $m\in \mathcal M(\T)$, $u_m$ as the unique solution of 
\begin{equation}\label{Eq:MainBil}\begin{cases}
{\partial_{t}}u_m-{\partial^2_{xx}} u_m=mu_m+ f(t,x,u_m)&\text{ in }\TT\,, 
\\ u_m(0,\cdot)=u^0&\text{ in }\T.\end{cases}
\end{equation} 
By \cite[Theorem 5.2, Chapter 1]{Pao} there exists a unique solution $u_m$ of \eqref{Eq:MainBil} (we refer to lemma \ref{Pr:Regularity} for  further regularity information about $u_m$). 
{
\begin{example}[A typical example of such a non-linearity]\label{Ex} While, taking $f\equiv 0$, we recover the standard case of the bilinearly controlled heat equation, let us point to a central example in spatial ecology, that of monostable equations; among them, the most standard one is the logistic-diffusive equation, which  would correspond to taking $f:u\mapsto -u^2$. However, while this non-linearity satisfies \eqref{H1}-\eqref{H2}, it violates the Lipschitz condition of \eqref{H3}. A way to overcome this difficulty is to extend $f$, outside of $(0,\max(||u^0||_{L^\infty},\kappa))$, to a smooth,  globally Lipschitz function that still satisfies \eqref{H2}. Indeed, \eqref{H2} ensures that any solution $u$ of \eqref{Eq:MainBil} remains bounded by  $\max\left(\Vert u^0\Vert_{L^\infty}\,, \kappa\right)$ (see lemma \ref{Le:ubdd}), so that this modification of $f$ outside of $(0;\max\left(\Vert u^0\Vert_{L^\infty}\,, \kappa\right))$ does not change the equation \eqref{Eq:MainBil}. We also refer to Remark \ref{Cp}.
\end{example}
}
\paragraph{Optimisation problem in the parabolic context: time-constant controls}
We consider fairly general functionals that we seek to optimise. To define this functional, we consider two functions $j_1\,, j_2$, a time horizon $T>0$ and we define 
\begin{equation}\mathcal J:\mathcal M(\T)\ni m\mapsto \iint_\TT j_1(t,x,u_m)+\int_\T j_2\left(x,u_m(T,\cdot)\right).\end{equation}
We mentioned earlier the crucial role of the monotonicity of the functional $\mathcal J$, which hinges on that of $j_1$ and $j_2$; we refer to section \ref{Se:Bib} for further comments. We thus assume that $j_1\,, j_2$ satisfy 
{\begin{multline}\tag{$\bold H_{\mathcal J}$}\label{Eq:HJ}
\text{ $j_1$ is $\mathscr C^1$ in $(t,x)$ and $\mathscr C^2$ in $u$,  $j_2$ is $\mathscr C^2$ in $(x,u)$, }
\\\text{$\partial_uj_1\,, \partial_u j_2\geq 0$ in $(0;T)\times \T\times \R_+$ and }\begin{cases}\text{ either $\partial_u j_1>0\text{ in }(0;T)\times \T\times (0;+\infty)$   }
\\\text{or $\partial_u j_2(t,x,\cdot)>0\text{ in }(0;T)\times \T\times (0;+\infty)$,}\end{cases}\\\text{ $ \forall K\in \R_+^*\,, \sup_{(t,x)\in \TT\,, u\in[0,K]}\underset{\alpha=0,1,2}{{\max}} \left|{\partial_t}j_1(t,x,u)\right|+ \left|{\partial^\alpha_{u^\alpha}} j_1(t,x,u)\right|+\left|{\partial^\alpha_{u^\alpha}}j_2(x,u)\right|<\infty, $}
\end{multline}} and we say (with a slight abuse of notation, identifying $\mathcal J$ with $(j_1,j_2)$) that $\mathcal J$ satisfies \eqref{Eq:HJ}.

In particular, we may choose 
\[j_1(t,x,u)=u^\alpha\text{ and }j_2(x,u)=u^\beta\] for $\alpha\,, \beta>0$, or $j_1=\p_1(u)\,, j_2=\p_2(u)$ with $\p_1\,, \p_2$ smooth and non-decreasing and at least one with positive derivative on $\R_+^*$, or $j_1(u)=\psi(x)u\,, j_2(u)=\psi_2(x)u$ with $\psi_1\,, \psi_2\geq 0$ and ${\max}\left(\inf |\psi_1|\,, \inf|\psi_2|\right)>0$.
The parabolic optimisation problem reads 

\begin{equation}\label{Eq:Pv}\tag{$\bold P_{\mathrm{parab}}$}\fbox{$\displaystyle
\max_{m\in \mathcal M(\T)}\mathcal J(m)$}\end{equation}
\begin{remark}[Existence of maximisers]
The existence of a solution of the variational problems \eqref{Eq:Pv} is a standard consequence of the direct method in the calculus of variations.
\end{remark}

\subsubsection{Main result}\label{Se:MainParab}
We state our main result:
\begin{theorem}\label{Th:Bgbg}
Assume $f$ satisfies \eqref{H1}-\eqref{H2}-\eqref{H3}. Assume $\mathcal J$ satisfies \eqref{Eq:HJ}. Any solution $m^*$ of \eqref{Eq:Pv} is bang-bang: there exists $E\subset \T$ such that 
\[m^*=\mathds 1_E.\]
\end{theorem}
The proof of this theorem is the core of this paper, and takes up the entire section \ref{Se:ProofTh1}.

\paragraph{An application to a spatial ecology problem}
We can apply Theorem \ref{Th:Bgbg} to the following spatial ecology problem. Consider, for any $m\in \mathcal M(\T)$, the logistic-diffusive equation 
\begin{equation}\begin{cases}\label{Eq:LDE}
{\partial_{t}} \theta_m-{\partial^2_{xx}} \theta_m=\theta_m\left(m-\theta_m\right)&\text{ in }\TT\,, 
\\ \theta_m(0,\cdot)=\theta^0\in \mathscr C^2(\T)&\text{ in }\T\,,
\\\inf_\T \theta^0>0. \end{cases}\end{equation} 
{
\begin{remark}\label{Cp}
As noted earlier, we need to ensure that the non-linearity $f:u\mapsto -u^2$ satisfies the assumptions \eqref{H1}-\eqref{H2}-\eqref{H3}. Here, using the maximum principle, we can obtain an \emph{a priori} bound: for any solution $\theta_m$ of \eqref{Eq:LDE}, we have $0\leq \theta_m\leq \max\left(\Vert u^0\Vert_{L^\infty}\,, 1\right)$. Consequently, we can take any smooth, Lipschitz extension $\tilde f$ of the map $f:u\mapsto -u^2$ outside of $(0; \max\left(\Vert u^0\Vert_{L^\infty}\,, 1\right))$, and \eqref{Eq:LDE} remains unaffected by this change of non-linearity.
\end{remark}
}In this context, the control $m$ accounts for a resources distribution available to a population, the density of which is the function $\theta_m$. A relevant query is to solve the optimisation problem 
\begin{equation}\label{Size}\sup_{m\in \mathcal M(\T)}\iint_\TT  \theta_m(t,\cdot)+\int_\T \theta_m(T,\cdot),\end{equation}for some time horizon $T>0$. This problem is the parabolic counterpart of a related elliptic optimisation problem that was intensively studied in the past few years, see  section \ref{Se:Bib} and \cite{DeAngelis2015,Heo2021,InoueKuto,LiangZhang,LouInfluence,Lou2008,Mazari2020,MNPCPDE,MRBSIAP,LouNagaharaYanagida,NagaharaYanagida}. In the elliptic case, the bang-bang property for optimisers was, in particular, a question that drew a lot of attention \cite{Mazari2020,MRBSIAP,NagaharaYanagida} and was only recently settled in \cite{MNPCPDE}. In the parabolic setting, \emph{i.e.} for problem \eqref{Size}, we refer, for instance, to the recent \cite{BINTZ2020} for the derivation of optimality conditions, as well as for some numerical simulations for a close variant of \eqref{Size}. Then, as a  corollary of Theorem \ref{Th:Bgbg} we obtain
\begin{corollary}
Any solution $m^*$ of \eqref{Size} is of bang-bang type.\end{corollary} The same conclusion holds for the two related problems
\[
\sup_{m\in \mathcal M(\T)}\int_\T \theta_m(T,\cdot)\,, \sup_{m\in \mathcal M(\T)}\iint_\TT \theta_m.\]

 It should be noted that, in \cite{BINTZ2020}, the case $m=m(t,x)$ is also considered. Our results do not hold in this case, as already underlined. We nonetheless refer to section \ref{Se:TDM} for generalisation of Theorem \ref{Th:Bgbg} to some classes of time-dependent controls.

We conclude this section on parabolic models with some remarks .
\paragraph{Some remarks on theorem \ref{Th:Bgbg}}
\begin{remark}[{Two examples of functionals}]\label{Re:Functional}
{Let us discuss two types of functionals that are of particular interest in control theory. The first class corresponds to the situation where the goal is to minimise the distance to a reference configuration at a final time or, in other words, } to solve an optimisation problem of the form \begin{equation}\label{TF}\inf_{m\in \mathcal M(\T)}\int_\T \vert u_m(T,\cdot)-y_{\mathrm{ref}}\vert^2,\end{equation} where $y_{\mathrm{ref}}$ is a target state. We would like to underline that such problems fall in our framework provided $y_{\mathrm{ref}}$ is large enough to ensure that, for any $m\in \mathcal M(\T)$ and any $T>0$, $y_{\mathrm{ref}}>u_m(T,\cdot)$. Indeed, as will be clear throughout the proof, the assumption that $\partial_u j_1(t,x,u)>0$ on $\R_+^*$ can be restricted to \[\forall (t,x)\in \TT\,,\partial_u j_1(t,x,u)>0\text{ in }\left(0,\sup_{m\in \mathcal M(\T)\,, T\in \R_+^*}\Vert u_m(T,\cdot)\Vert_{L^\infty}\right).\] It is the minimum requirement we can make, as we know that the solutions of some tracking-type problems are not bang-bang.

{Similarly,  the second class of functionals, dubbed ``tracking-type" functionals, consists in driving the state of the equation as close as possible to a reference configuration, this time averaging the distance over time. For instance, if we } consider the case of \eqref{Eq:LDE}  with $\theta^0<1$ we want to solve the optimisation problem \[\min_{m\in \mathcal M(\T)}\iint_\TT| \theta_m-1|^2. \]  This optimisation problem is equivalent to 
\[\max_{m\in \mathcal M(\T)} \left(-\iint_\TT| \theta_m-1|^2 \right).\] As, from the maximum principle, $\theta_m<1$, the map $j_1(x,u)=-|1-u|^2$ is increasing and has a positive derivative in $(0;1)$, whence we can apply theorem \ref{Th:Bgbg} to conclude that any minimiser of the initial problem is of bang-bang type. This example immediately generalises to the case where the target $y_{\mathrm{ref}}=1$ is replaced with any target $y_{\mathrm{ref}}\geq 1$ a.e.

 \end{remark}
 
 \begin{remark}[Regarding the $L^\infty$ constraints on $m$]\label{Re:Constraints}
 It should be noted that we may consider constraints of the form 
\[{-\kappa_0}\leq m\leq {\kappa_1}\] for two parameters ${\kappa_0\,, \kappa_1}\in \R$ (possibly non-positive), without changing the conclusions of the theorems. The only difference would be that a bang-bang $m$ would in that case be of the form $m={-\kappa_0}+({\kappa_1}+{\kappa_0})\mathds 1_E$.  Indeed, the proof relies on a second-order derivative argument that still holds in this case, as the key point is that $u_m(t,\cdot)$ is positive whenever $t>0$. For an interaction of the type $mu_m$ with a sign-changing $m$, this is still the case.
 \end{remark}
 
 \begin{remark}[Regarding the bilinearity of the control]\label{Re:Interaction}
 It is also worth noting that our method is flexible enough to cover the case of other interactions, of the form $m\p(u_m)$, for suitable non-linearities $\p$. In the conclusion, see theorem \ref{Th:BgbgP}, section \ref{Se:Interaction}, we give a sketch of proof for a version of theorem \ref{Th:Bgbg} for this type of interactions. The main condition on $\p$ to ensure that the bang-bang property holds is that $\p'/\p>0$ in $\R_+^*$, which is still sufficient to cover the case of the optimisation of the carrying capacity, where one works with the equation $\partial_t y_m-{\partial^2_{xx}} y_m=y_m(1-my_m)$. This last example is motivated by \cite{DeAngelis2015}.
 \end{remark}

\begin{remark}[Regarding the time dependency of the control]\label{Re:TDDD}
Our method also allows to cover a form of time-dependency of the control. If we consider, rather than $m(x)u_m(t,x)$, an interaction of the form $ u_m(t,x)\sum_{i=1}^N \phi_i(t)m_i(x)$, where the $\phi_i$ are bounded, smooth functions of time, then the bang-bang property holds. We refer to theorem \ref{Th:BgbgN} , section \ref{Se:TDM}. In the conclusion, see section \ref{Se:TimeDependent}, we explain the difficulty in obtaining the same results for general time-dependent controls. 

\end{remark}

\begin{remark}[Regarding the one-dimensional setting]
The reason we work in the one-dimensional torus is a technical one; while the dimension of the torus is not problematic (see section \ref{ff} of the conclusion), the space periodicity enables us to carry out rigorous two-scale expansions in the following setting: a key part of the proof is the study of the solution $\dot u$ of an equation of the form 
\[ \partial_t \dot u-{\partial^2_{xx}} \dot u=V(t,x)\sum_{k\gg1}\alpha_k \cos(kx),\] and we study $\dot u$ by providing an explicit expansion  as all the Fourier {indices} $k\gg 1$ are very large. While, in more general domains, we may replace the right-hand side in the equation above with $\sum_{k\gg 1} \alpha_k \psi_k(x)$ for some eigenfunctions $\psi_k$ of the {Laplacian}, it is not yet clear how we may reach the necessary conclusion.
\end{remark}

\subsection{Comments on the proof of Theorem \ref{Th:Bgbg}}\label{Se:Comment}

{The} starting point of our method is an idea we introduced in \cite{MNPCPDE} in the elliptic context. 
{Let us briefly recap the context of \cite{MNPCPDE}. Consider an elliptic bilinear optimal control problem of the form \[\sup_{m\in \mathcal M(\T)}\int_\T j(z_m)(=J_{\mathrm{ell}}(m))\]} where $\mathcal M(\T)$ is defined in \eqref{Eq:Adm}, subject to {the elliptic equation \[-\Delta z_m=mz_m+g(x,z_m).\] }We proved that, under suitable assumptions on the non-linearity $g$ and provided $j$ was increasing, any optimiser $m^*_{\mathrm{ell}}$ was a bang-bang function. {In order to do so, we argued by contradiction, picking a non bang-bang maximiser $m^*_{\mathrm{ell}}$, so that the set $\omega^*:=\{0<m^*_{\mathrm{ell}}<1\}$ has positive measure. This implies that for any admissible perturbation $h$ at $m^*_{\mathrm{ell}}$ supported in $\omega^*$, the first-order Gateaux-derivative $\dot J_{\mathrm{ell}}(m^*_{\mathrm{ell}})[h]$ at $m^*_{\mathrm{ell}}$ in the direction $h$ satisfies 
\[ \dot J_{\mathrm{ell}}(m^*_{\mathrm{ell}})[h]=0.\] We then used a second-order argument. Structurally, if $z_m>0$ and if $j'>0$ is increasing we can prove the existence of two constants $\alpha\,, \beta>0$ such that the second order derivative $\ddot{J}_{{\mathrm{ell}}}(m)[h,h]$ of $J_{\mathrm{ell}}$ at an admissible control $m$ in the direction of a perturbation $h$ satisfies the lower estimate
\begin{equation}\label{eez}\ddot{J}_{{\mathrm{ell}}}(m)[h,h]\geq \alpha \int_\T |\n \dot z_m|^2-\beta\int_\T \dot z_m^2,\end{equation} where $\dot z_m$ is the derivative of the map $m\mapsto z_m$ at $m$ in the direction $h$. What matters is that $\dot z_m$ solves a linear equation of the form
\[ \mathcal L_m\dot z_m(:=-\Delta \dot z_m-V_mz_m)=hz_m\] where $V_m=\partial_z g(x,z_m)$ is a bounded potential. Finally, to conclude on \eqref{eez}, we note that elliptic regularity is crucial in deriving it, and is much easier in the elliptic case than in the parabolic one under consideration here.

Going back to the proof of the bang-bang property, and recalling that by contradiction we consider a non bang-bang optimiser $m^*_{\mathrm{ell}}$, since, for any perturbation $h$ supported in $\omega=\{0<m^*_{\mathrm{ell}}<1\}$, $\dot J(m^*_{\mathrm{ell}})[h]=0$, it suffices to build a perturbation $h$ supported in $\omega$ and such that 
\[ \ddot J_{\mathrm{ell}}(m^*)[h,h]>0.\]
 Given the bound \eqref{eez}, it is enough to find a suitable $h$ supported in $\omega$ such that $\int_\T|\n \dot z_{m_{\mathrm{ell}}^*}|^2\gg \int_\T \dot z_{m_{\mathrm{ell}}^*}^2$. We then use the {eigenpairs} $\{\lambda_k,\psi_k\}_{k\in \N}$ of the operator $\mathcal L_{m_{\mathrm{ell}}^*}$ to build an $h$ such that, in the eigenfunction basis, the only non-zero Fourier modes of $\dot z_{m_{\mathrm{ell}}^*}$ are high order Fourier modes:\[\dot z_m=\sum_{k\geq K}\alpha_k \psi_k\text{ for some large integer $K$}.\] This leads to 
 \[ \int_\T|\nabla \dot z_m|^2\geq \lambda_K \int_\T \dot z_m^2.\] Taking $K$ large enough (depending on $\alpha\,, \beta$) we can then reach a contradiction. In other words, the crux of the matter was the construction and analysis of highly oscillating perturbations.}

%

In the case of parabolic equations, obtaining a lower order estimate {akin to \eqref{eez}} on the second order derivative of the functional requires some time regularity on the controls $m$ (hence the assumption that $m$ does not depend on time in theorem \ref{Th:Bgbg}). However, even with such an estimate at hand, the problem is still not solved. Indeed, in the parabolic case, the structure of the equation on $\dot u_m$ (the Gateaux-derivative of the map $m\mapsto u_m$ at $m$ in an admissible direction $h$) is rather of the form \[\partial_t\dot u_m-{\partial^2_{xx}} \dot u_m=V_m(t,x)\dot u_m+hu_m(t,x),\] with a time-varying potential $V_m$, and the lower estimate assumes the form
\[ \ddot{\mathcal J}(m)[h,h]\geq \alpha \iint_\TT |{\partial_x} \dot u_m|^2-\beta\iint_\TT \dot u_m^2-\gamma \int_T \dot u_m(T,\cdot)^2.\]
We refer to estimate \eqref{Eq:EstDDJ}, proposition \ref{Pr:Cruise}, for a precise statement.{This is the first crucial step of the proof.} {Thus it is natural to try and argue by contradiction, considering a non bang-bang optimiser $m^*$, so that $\omega=\{0<m^*<1\}$ has positive measure, and finding a perturbation $h$ supported in $\omega$ and satisfying 
\[\iint_\TT |\partial_x \dot u_m|^2\gg \iint_\TT \dot u_m^2+\int_\T \dot u_m(T,\cdot)^2.\]
}
 But even in the one-dimensional case, finding a perturbation $h$ such that, for a fixed (and large) integer $K\in \N$, we have $\dot u_m=\sum_{k\geq K}\alpha_k\phi_k(t)\cos(kx)$ proves impossible because the potential $V_m$ varies in time. Thus was have to resort to some two-scale asymptotic expansions in order to attain an approximation \[\dot u_m\approx\sum_{k\geq K}\alpha_k\phi_k(t)\cos(kx)\] that is strong enough, see proposition \ref{Pr:Rosecroix}.{This is the second major part of the proof and takes up most of the article. It}  is in part inspired by \cite{MNT} and by seminal works on two-scale expansions \cite{Allaire}, but requires some fine improvements to be better suited to our needs. The need for these improvements essentially comes from the fact that the potential $V_m$ is merely $W^{2,p}$ in space, and not $\mathscr C^2$, as is customary in such queries.

\subsection{Relationship with some shape optimisation problems}\label{Se:SO}
There is another possible outlook on theorem \ref{Th:Bgbg}  that relates {its} conclusions to (non-)existence results for PDE constrained shape optimisation problems.  For any subset $E\subset \T$, we may define the shape functional 
\[ \mathcal F(E):=\mathcal J(\mathds 1_E),\] and investigate the shape optimisation problem 
\begin{equation}\label{Eq:PvSO}
\sup_{E\subset \T\,, \mathrm{Vol}(E)=V_0} \mathcal F(E).\end{equation}
For this type of optimisation problems, it is usually very difficult to obtain an existence property. The most general result is that of Buttazzo and {Dal Maso}, in the seminal {paper} \cite{BDM}, which states that, if the functional $\mathcal F$ is increasing with respect to the set inclusion, and is moreover upper semi-continuous for the $\gamma$-convergence of sets, then an optimal set $E^*$ exists. In theorem \ref{Th:Bgbg}, the monotonicity (which in turn hinges on that of $(j_1,j_2)$) plays a crucial role, and we prove that under assumption \eqref{Eq:HJ}, $\mathcal J$ is indeed increasing. However, $\mathcal F$ is not continuous for the $\gamma$-convergence of sets, which thus prevents using the result of \cite{BDM}. This is well-known, and usually leads to considering  the relaxation of the class of admissible sets $\mathcal A:=\{\mathds 1_E\,, E\subset \T\,, \mathrm{Vol}(E)=V_0\}$ in the weak $L^\infty-*$ topology; this relaxation exactly corresponds to \eqref{Eq:Adm}, and the relaxed version of \eqref{Eq:PvSO} is \eqref{Eq:Pv}. In this way, theorem \ref{Th:Bgbg} states that every solution of \eqref{Eq:Pv} belongs to $\mathcal A$, and, consequently, that \eqref{Eq:PvSO} has a solution. 

This remark was also one of the motivation for the present work, which continues a series of papers devoted to establishing existence results for some shape optimisation problems, see \cite{MPRobin,MNPCPDE}.

\subsection{Bibliographical references and comments}\label{Se:Bib}

Bilinear optimal control problems and the optimal control of semilinear parabolic models are present in a very large number of fields of applied mathematics. It is impossible to give an exhaustive list of contributions, but we single out a few that we think are closely related to our queries.
\subsubsection{Elliptic bilinear optimal control problem}

\paragraph{Spectral optimisation problems}Let us begin with a spectral optimisation problem. In this setting, one aims at minimising the first eigenvalue $\lambda(m)$ of the operator $-\Delta-m$ in a smooth bounded domain $\O\subset \R^d$, endowed with certain boundary conditions, with respect to admissible controls $m$ that satisfy $L^\infty$ and $L^1$ constraints. The reason this problem is bilinear is that the state equation assumes the form 
\[-\Delta z_m=mz_m+\lambda(m)z_m.\]
The study of the minimisation problem of $\lambda(m)$ with respect to $m$ originates in spatial ecology consideration see \cite{CC1,CantrellCosner1,MR1105497,CC4,CantrellCosner,31806255e7f648d5b65ff02c30c4f539,CaubetDeheuvelsPrivat,JhaPorru,KaoLouYanagida,LamboleyLaurainNadinPrivat,MazariThese,MNPARMA,Shigesada} and the references therein. For such problems, the bang-bang property is usually immediate \cite{KaoLouYanagida} and can be deduced from the concavity of the functional at hand or from classical tools such as the bathtub principle. Similarly, following \cite{BHR}, the geometric properties of optimisers have been thoroughly analysed, and are by now well understood; the main tool for this query is that of rearrangements, and a key point is that the functional is energetic. We refer to \cite{LamboleyLaurainNadinPrivat} for up to date results in this direction, as well as to the survey \cite{MNPChapter}.  

\paragraph{A non-energetic elliptic bilinear optimal control problem}
A problem which displays the rich behaviour of elliptic bilinear optimal control problems is that of the total population size in logistic-diffusive equations. In this setting, the PDE writes
\[-\mu\Delta \theta_m=\theta_m(m-\theta_m)\,, \theta_m\geq 0\,, \theta_m\neq 0\] with Neumann or Robin boundary conditions. The control $m$ is assumed to satisfy $L^\infty$ and $L^1$ constraints. The functional to optimise is $J:m\mapsto \int_\O \theta_m$. For modelling issues, we refer to \cite{LouInfluence} and the references therein. Obtaining the bang-bang property for its maximisers is surprisingly difficult. In \cite{Mazari2020,NagaharaYanagida} this bang-bang property is proved under several restrictive assumptions. In \cite{MNPCPDE}, we introduced a new method to prove this property without these {assumptions}. We refer to section \ref{Se:Comment} above to see why the method of \cite{MNPCPDE} does not apply in the parabolic context. Regarding the geometric features of optimisers, it was proved in \cite{Heo2021,MRBSIAP} that the $BV$-norm of optimisers blows up as $\mu \to 0^+$; in \cite{MNPCPDE}, this blow-up rate is quantified. It would be very interesting, in the context of parabolic models, to obtain such qualitative information about the geometry of maximisers.

\subsubsection{Optimal bilinear control of parabolic equations} Since we adopt, in the present paper, the point of view of optimal control, we merely indicate that there is a branch of research devoted to the question of \emph{bilinear controllability} (\emph{i.e.} is it possible to reach an exact state using a bilinear control ?); we refer the interested reader to  \cite{Alabau,Beauchard2010,Floridia} and the references therein.  In the field of bilinear optimal control problems, let us first point to \cite{Fister,GuillnGonzlez2020} for the study of bilinear control problems in connection with chemotaxis or chemorepulsion; another very interesting example of such a problem is studied in \cite{2021}. In it, the authors study an optimal control problem for brain tumor growth. Although their bilinear control only depends on time (\emph{i.e.} their function $m$ satisfies $m=m(t)$, which is exactly the type of case not covered in the present contribution), some emphasis is put, through numerical simulations, on the bang-bang property. 

A very relevant reference for the type of problems we are studying is the recent \cite{BINTZ2020}, in which the exact problem of optimisation of the total population size for parabolic logistic-diffusive equations is studied under the same type of constraints we have here. The optimality conditions are derived, and several numerical simulations are carried out.

\section{Proof of {Theorem} \ref{Th:Bgbg}}\label{Se:ProofTh1}
We break this section down in several parts: first, we give a basic positivity estimate on $u_m$. Then we compute the first and second order Gateaux derivative of the criterion by the use of an adjoint state. Moreover, we give all the regularity information that is needed, and we use them to obtain a lower estimate on the second-order derivative. Finally, we provide a fine analysis of this second order derivative using two scale asymptotic expansions.

We first have the following basic estimate on $u_m$:
\begin{lemma}\label{Le:ubdd}{Let $\kappa$ be given by Assumption \eqref{H2}}.
There holds
\[\forall m\in \mathcal M(\T)\,,  0<\inf_{ \TT}u_m\leq \Vert u_m\Vert_{L^\infty(\TT)}\leq \max\{\Vert u_0\Vert_{L^\infty},\kappa\}.\]
\end{lemma}
As this lemma is a straightforward consequence of the maximum principle, we prove it in appendix \ref{Ap:ubdd}.

\subsection{Computation of first and second order Gateaux derivatives \emph{via} an adjoint state}
In this section we analyse the {first and second order} Gateaux derivatives of the {functional}, and then comment on its monotonicity.
It is standard to see that the map 
$$m\mapsto u_m$$ is  twice Gateaux differentiable. For a given $m\in \mathcal M(\T)$ and an admissible\footnote{\label{footnotehAdm}The wording ``admissible perturbation'' means that $h$ belongs to the tangent cone to the set {$\mathcal M(\T)$} at $\beta$.
It corresponds to the set of functions $h\in L^\infty(\O)$ such that, for any sequence of positive real numbers $\varepsilon_n$ decreasing to $0$, there exists a sequence of functions $h_n\in L^\infty(\O)$ converging in $L^2(\O)$ to $h$ as $n\rightarrow +\infty$, and $\beta+\varepsilon_nh_n\in{\mathcal M(\T)}$ for every $n\in\N$.} perturbation $h$ at $m$, we call $\dot u_m[h]$ (resp. $\ddot u_m[h]$) the first (resp. the second) order Gateaux derivative in the direction $h$. When no ambiguity is possible, we use the notation $\dot u_m$ (resp. $\ddot u_m$) for the first (resp. second) order Gateaux derivative at $m$ in the direction $h$.  It is straightforward to see that $\dot u_m$ solves 
\begin{equation}\label{Eq:Dotu}
\begin{cases}
{\partial_{t}}\dot u_m-{\partial^2_{xx}} \dot u_m-\dot u_m\left(m+\left.{\partial_u}f\right|_{u=u_m}\right)=h u_m&\text{ in }\TT,
\\ \dot u_m(0,\cdot)\equiv 0 &\text{ in }\T.
\end{cases}
\end{equation}
Similarly we obtain, for the first-order Gateaux derivative of $\mathcal J$ at $m$ in the direction $h$ the expression 
\begin{equation}\dot{\mathcal J}(m)[h]=\iint_\TT \dot u_m\left.{\partial_u}j_1\right|_{u=u_m}+\int_\T \dot u_m\left.{\partial_u}j_2\right|_{u=u_m(T,\cdot)}.\end{equation}
We define 
\begin{equation}V_m:=m+\left.{\partial_u}f\right|_{u=u_m}{\in L^\infty(\TT)\cap L^2(0,T;L^2(\T))}.\end{equation} and introduce the function $p_m$ as the solution of the backward parabolic equation 

\begin{equation}\label{Eq:Adjoint}\begin{cases}
{\partial_{t}}{ p_m}+{\partial^2_{xx}} p_m+V_m p_m=-\left.{\partial_u}j_1\right|_{u=u_m}&\text{ in }\TT\,, 
\\ p_m(T,\cdot)=\left.{\partial_u}j_2\right|_{u=u_m} &\text{ in }\T.
\end{cases}\end{equation}
 Multiplying \eqref{Eq:Adjoint} by $\dot u_m$ and integrating by parts, we obtain
 \begin{equation}\label{Eq:WDJ}
\dot{\mathcal J}(m)[h]= \int_\T\left.{\partial_u}j_2\right|_{u=u_m(T,\cdot)}\dot u_m+\iint_\TT \left.{\partial_u}j_1\right|_{u=u_m}\dot u_m=\iint_\TT h u_m p_m.\end{equation}
Let us now comment on the monotonicity of the functional, which shall play a crucial role in the forthcoming analysis. Of course, none of the computations above require that $h$ be admissible, and we may take, for $h$, a non-negative function, as the constraints (and hence the admissibility of $h$) only play a role in the derivation of optimality conditions. By monotonicity we mean the following property: 
\[\forall m\in \mathcal M(\T)\,, \forall h\in L^\infty(\O)\,, h\geq 0\Rightarrow \dot{\mathcal J}(m)[h]\geq0.\]
Given \eqref{Eq:WDJ} and lemma \ref{Le:ubdd}, this monotonicity property actually holds if $p_m$ itself is positive. This is where assumption \eqref{Eq:HJ} comes into play:
\begin{lemma}\label{Le:RegAdjoint}If $\mathcal J$ satisfies \eqref{Eq:HJ} then {the solution $p_m$ of \eqref{Eq:Adjoint} satisfies }
$p_m\in W^{2,2}(\TT)$ and, for any $\e>0$, 
$$\inf_{[0,T-\e]\times \T}p_m>0.$$
\end{lemma}

 \begin{proof}[Proof of lemma \ref{Le:RegAdjoint}]
 If we set $q_m(t,\cdot):=p_m(T-t,\cdot)$ it appears that $q_m$ solves
 \[\begin{cases}
{{\partial_{t}} q_m}-{\partial^2_{xx}} q_m-V_m q_m=\left.{\partial_u}j_1\right|_{u=u_m}\text{ in }\TT\,, 
\\ q_m(0,\cdot)=\left.{\partial_u}j_2\right|_{u=u_m} &\text{ in }\T.
\end{cases}\]
If $\partial_uj_2>0$ then, as $\partial_u j_1\geq 0$ the conclusion follows from the strong maximum principle. Likewise, if on the other hand we merely have $\partial_u j_2\geq 0$ then, as $\partial_uj_1>0$ in this case, we obtain the conclusion by the strong maximum principle.
 \end{proof}
 
 We now move on to the computation of the second order Gateaux derivative of the functional at hand.  The second order derivative of $m\mapsto u_m$ in the direction $h$ solves
 \begin{equation}\label{Eq:Ddotu}
 \begin{cases}
 {{\partial_{t}} \ddot u_m}-{\partial^2_{xx}} \ddot u_m-V_m\ddot u_m=2h \dot u_m+\left.{\partial^2_{uu}}f\right|_{u=u_m} \left(\dot u_m\right)^2&\text{ in }\TT,
\\ \ddot u_m(0,\cdot)\equiv 0&\text{ in }\T
 \end{cases}
 \end{equation}
We also have, for the second order Gateaux derivative of $J$ at $m$ in the direction $h$, the following expression:
 \begin{multline}\ddot{\mathcal J}(m)[h,h]=\int_\T \dot u_m^2(T,\cdot) \left.{\partial^2_{uu}}j_2\right|_{u=u_m(T,\cdot)}+\int_\T \ddot u_m(T,\cdot)\left.{\partial_u}j_2\right|_{u=u_m(T,\cdot)}
\\+\iint_\TT \dot u_m^2 \left.{\partial^2_{uu}}j_1\right|_{u=u_m}+\iint_\TT \ddot u_m\left.{\partial_u}j_1\right|_{u=u_m}.
 \end{multline}
 We use the adjoint state $p_m$ again: multiplying \eqref{Eq:Adjoint} by $\ddot u_m$ and integrating by parts, we obtain 
 \begin{multline}
 2\iint_\TT h\dot u_mp_m+\iint_\TT\left.{\partial^2_{uu}}f\right|_{u=u_m}(\dot u_m)^2p_m\\=\int_\O \ddot u_m\left.{\partial_u}j_2\right|_{u=u_m(T,\cdot)}+\iint_\TT \ddot u_m \left.{\partial_u}j_1\right|_{u=u_m}
 \end{multline}
 so that 
   \begin{multline}\ddot{\mathcal J}(m)[h,h]=\int_\T \dot u_m^2(T,\cdot) \left.{\partial^2_{uu}}j_2\right|_{u=u_m(T,\cdot)}
  \\+\iint_\TT \dot u_m^2 \left.{\partial^2_{uu}}j_1\right|_{u=u_m}+ 2\iint_\TT h\dot u_mp_m+\iint_\TT p_m\dot u_m^2\left.{\partial^2_{uu}}f\right|_{u=u_m}.
 \end{multline} Rearranging the terms, we get
 
   \begin{multline}\ddot{\mathcal J}(m)[h,h]=2\iint_\TT h\dot u_mp_m+\int_\T \dot u_m^2(T,\cdot) \left.{\partial^2_{uu}}j_2\right|_{u=u_m(T,\cdot)}
  \\+\iint_\TT \dot u_m^2\left( \left.{\partial^2_{uu}}j_1\right|_{u=u_m} +\left.{\partial^2_{uu}}f\right|_{u=u_m}p_m\right).
 \end{multline} 
Let us focus on the term 
\begin{equation}\label{Eq:RTS}\iint_\TT h\dot u_m p_m.\end{equation} {By Lemma \ref{Le:ubdd}, we know that $u_m>0$.} From \eqref{Eq:Dotu} we rewrite 
$$h=\frac{{\partial_{t}}\dot u_m-{\partial^2_{xx}} \dot u_m-\dot u_mV_m}{u_m}.$$ Let us define 
$$\Psi_m:=\frac{p_m}{u_m}.$$ Plugging this expression in \eqref{Eq:RTS} we obtain 
\begin{align*}
\iint_\TT h\dot u_m  p_m&=\iint_\TT \Psi_m \dot u_m \left({\partial_{t}} \dot u_m -{\partial^2_{xx}} \dot u_m -\dot u_m V_m\right)
\\&=\frac12\iint_\TT \Psi_m{\partial_{t}}\dot u_m^2
\\&+\iint_\TT \Psi_m \left|{\partial_x} \dot u_m \right|^2
\\&+\iint_\TT \dot u_m  \cdot {\partial_x} \Psi_m\cdot{\partial_x} \dot u_m 
\\&-\iint_\TT \Psi_mV_m\dot u_m ^2
\\&=-\frac12\iint_\TT( {\partial_{t}}{ \Psi_m})\dot u_m ^2+\frac12\int_\T \Psi_m(T,\cdot)\dot u_m ^2
\\&+\iint_\TT \Psi_m \left|{\partial_x} \dot u_m \right|^2-\frac12\iint_\TT \dot u_m ^2 {\partial^2_{xx}} \Psi_m -\iint_\TT \Psi_mV_m \dot u_m ^2
\\&=\frac12\int_\T \Psi_m(T,\cdot)\dot u_m ^2+\iint_\TT \Psi_m \left|{\partial_x} \dot u_m \right|^2
\\&-\frac12\iint_\TT ({\partial_{t}} \Psi_m)\dot u_m ^2-\frac12\iint_\TT \dot u_m ^2 {\partial^2_{xx}} \Psi_m -\iint_\TT \Psi_mV_m \dot u_m ^2
\\&=\frac12\int_\T \Psi_m(T,\cdot)\dot u_m ^2+\iint_\TT \Psi_m \left|{\partial_x} \dot u_m \right|^2
\\&+\iint_\TT\dot u_m ^2\left(-\frac12{\partial_{t}} \Psi_m-\frac12 {\partial^2_{xx}} \Psi_m - \Psi_mV_m \right).
\end{align*}
With$$\mathscr Z_m:=-\frac12{\partial_{t}} \Psi_m-\frac12 {\partial^2_{xx}} \Psi_m - \Psi_mV_m +\left.{\partial^2_{uu}}j_1\right|_{u=u_m}+\left.{\partial^2_{uu}}f\right|_{u=u_m}p_m,$$ 
the second order derivative writes
\begin{multline}\label{Eq:MB}\ddot{\mathcal J}(m)[h,h]=\int_\T \dot u_m^2(T,\cdot) \left.{\partial^2_{uu}}j_2\right|_{u=u_m(T,\cdot)}+\frac12\int_\T \Psi_m(T,\cdot)\dot u_m^2\\+\iint_\TT \Psi_m \left|{\partial_x}\dot u_m\right|^2+\iint_\TT \mathscr Z_m \dot u_m^2.\end{multline} We analyse this expression further in the next section.

\subsection{Lower estimate on the second order Gateaux derivative of $\mathcal J$}

We now prove the following lower estimate on this second order Gateaux derivative:

\begin{proposition}\label{Pr:Cruise}
Let $\e>0$ be arbitrarily small. There exist three positive constants $\alpha=\alpha(\e)\,,\beta\,, \gamma>0$ such that, for any admissible perturbation $h$ at $m$ \corr{and for any $\delta>0$}, there holds 
\begin{multline}\label{Eq:EstDDJ}
\ddot{\mathcal J}(m)[h,h]\geq  \alpha\corr{(1-\delta)}\iint_\TeT |{\partial_x} \dot u_m|^2-\corr{\alpha\delta\iint_{(T-\e;T)\times \T}|\partial_x\dot u_m|^2}
\\-\beta\corr{\left(1+\frac1\delta\right)}\iint_\TT \dot u_m^2-\gamma\int_\T \dot u_m^2(T,\cdot).\end{multline}\end{proposition}
 Proving this proposition requires some additional regularity on $u_m\,, p_m$. This is where the regularity of $m$ in time (here, $m$ is constant in time) is crucial. We gather these regularity properties in the following proposition:
 
\begin{proposition}\label{Pr:Regularity}
For any $m\in \mathcal M(\T)$, for any $p\in [1;+\infty)$, there exists a constant $\mathfrak M_{m,p}$ such that
\begin{multline}
\sup_{t\in [0,T]}\left\Vert \partial_t u_m(t,\cdot)\right\Vert_{L^p(\T)}+\sup_{t\in[0,T]}\left\Vert  u_m(t,\cdot)\right\Vert_{W^{2,p}(\T)}\\+\sup_{t\in [0,T]}\left\Vert \partial_t p_m(t,\cdot)\right\Vert_{L^p(\T)}+\sup_{t\in[0,T]}\left\Vert  p_m(t,\cdot)\right\Vert_{W^{2,p}(\T)}\leq \mathfrak M_{m,p}.
\end{multline}In particular, by Sobolev embeddings, there exists a constant $\mathfrak N$ such that 
\[\sup_{t\in (0,T)}\Vert u_m(t,\cdot)\Vert_{\mathscr C^1(\T)}\,, \sup_{t\in (0,T)}\Vert p_m(t,\cdot)\Vert_{\mathscr C^1(\T)}\leq \mathfrak N.\]\end{proposition}
The proof of this proposition is standard in the regularity theory of parabolic equations; it can be derived from classical $L^p$  estimates (see for instance \cite[Theorem 7.32, pp. 182-183]{Lieberman}) but the setting we are working in allows for a quicker proof, that we give in appendix \ref{Ap:Regularity}.
 
 With these regularity estimates we can prove proposition \ref{Pr:Cruise}
 \begin{proof}[Proof of proposition \ref{Pr:Cruise}]

First off, from lemmas \ref{Le:ubdd} and \ref{Le:RegAdjoint} we have that \[\forall \e>0\,, \inf_{(0,T-\e]\times \T}\Psi_m>0.\] Hence there exists a constant $\alpha=\alpha(\e)>0$ such that

$$\iint_\TT \Psi_m|{\partial_x} \dot u_m|^2\geq \alpha \iint_{(0,T-\e)\times \T}|{\partial_x} \dot u_m|^2.$$
Second, since $j_2$ is  $\mathscr C^2$ in $u$ and $u_m$ is bounded, there exists a constant $\gamma>0$ such that 
$$\int_\T  \dot u_m^2(T,\cdot) \left.{\partial^2_{uu}}j_2\right|_{u=u_m(T,\cdot)}\geq-\gamma\int_\T  \dot u_m^2(T,\cdot) .$$

\paragraph{Estimates on $\mathscr Z_m$} 
As for $\mathscr Z_m$, we rewrite it
$$\mathscr Z_m=-\frac12{\partial_t} \Psi_m-\frac12 {\partial^2_{xx}} \Psi_m +\mathscr Y_m\quad \mathrm{ with }\quad \mathscr Y_m:= -\Psi_mV_m +\left.{\partial^2_{uu}}j_1\right|_{u=u_m}+\left.{\partial^2_{uu}}f\right|_{u=u_m}p_m.$$ 
Of course, 
\[\mathscr Y_m\in L^\infty.\]
Let us focus on the term
\[-{\partial_t}\Psi_m-{\partial^2_{xx}} \Psi_m.\]
We compute
\begin{equation}
{\partial_t}\Psi_m=\frac1{u_m}{\partial_t}p_m-\frac{p_m}{u_m^2}{\partial_t}u_m,
\end{equation}
and
\begin{equation}
{\partial^2_{xx}}\Psi_m=-2\frac1{u_m^2}{\partial_x}p_m\cdot{\partial_x}u_m+\frac1{u_m}{\partial^2_{xx}}p_m-\frac{p_m}{u_m^2}{\partial^2_{xx}}u_m+2\frac{p_m}{u_m^3}\left({\partial_x} u_m\right)^2.
\end{equation}
Setting $$\mathscr X_m:=-2\frac1{u_m^2}({\partial_x} p_m)\cdot({\partial_x} u_m)+2\frac{p_m}{u_m^3}\left({\partial_x} u_m\right)^2$$ we have 
$${\partial^2_{xx}}\Psi_m=\frac1{u_m}\partial^2_{xx}p_m-\frac{p_m}{u_m^2}{\partial^2_{xx}}u_m+\mathscr X_m.$$
From Proposition \ref{Pr:Regularity} we know that
$$\mathscr X_m\in L^\infty(\TT).$$
Hence,

\begin{align*}
{\partial_t}\Psi_m+{\partial^2_{xx}}\Psi_m&=\frac1{u_m}{\partial_t}p_m-\frac{p_m}{u_m^2}{\partial_t} u_m
\\&+\frac1{u_m}{\partial^2_{xx}}p_m-\frac{p_m}{u_m^2}{\partial^2_{xx}}u_m+\mathscr X_m
\\&=\mathscr X_m
\\&+\frac1{u_m}\left({\partial_t}p_m+{\partial^2_{xx}}p_m\right)
\\&-\frac{p_m}{u_m^2}\left( {\partial_t}u_m+{\partial^2_{xx}}u_m\right)
\\&=\mathscr X_m+\frac{p_m}{u_m}\left(-\left.{\partial_u}j_1\right|_{u=u_m}\right)
\\&-\frac{p_m}{u_m^2}\left(2{\partial^2_{xx}}u_m+mu_m+f(t,x,u_m)\right),
\end{align*}
whence, using the fact that $j_1\in \mathscr C^1$ and proposition \ref{Pr:Regularity},
the function $\tilde{\mathscr Z_m}$ satisfies
\begin{equation}\label{Eq:EstZm}
\forall p\in [1;+\infty)\,, \sup_{t\in [0,T]}\left\Vert {\mathscr Z}_m(t,\cdot)\right\Vert_{L^p(\T)}=:\mathfrak M(p)<\infty.\end{equation}
This allows to estimate the last term in the second-order Gateaux derivative: from the Sobolev embedding $W^{1,2}(\T)\hookrightarrow \mathscr C^0(\T)$ with constant $C_{\mathrm{sob}}$ and the Cauchy-Schwarz inequality we obtain 
\begin{align*}
\iint_\TT \mathscr Z_m \dot u_m^2&\geq -\int_0^T\Vert \dot u_m(t,\cdot)\Vert_{W^{1,2}(\T)}\cdot \Vert { \mathscr Z}_m(t,\cdot)\Vert_{L^2(\T)}\cdot \Vert \dot u_m(t,\cdot)\Vert_{L^2(\T)}dt
\\&=-\mathfrak M(2)C_{\mathrm{sob}}\int_0^T\Vert \dot u_m(t,\cdot)\Vert_{W^{1,2}(\T)}\cdot\Vert \dot u_m(t,\cdot)\Vert_{L^2(\T)}dt
\\&=-\mathfrak M(2)C_{\mathrm{sob}}\int_0^T\Vert{\partial_x}\dot u_m(t,\cdot)\Vert_{L^{2}(\T)}\cdot\Vert \dot u_m(t,\cdot)\Vert_{L^2(\T)}dt
\\&-\mathfrak M(2)C_{\mathrm{sob}}\iint_\TT \dot u_m^2.
\end{align*}
\begin{remark}\label{Argued}It may be argued that here we already use the fact that we are working in the one-dimensional setting, when using the Sobolev embedding $W^{1,2}\hookrightarrow \mathscr C^0$. However, this can be very well extended to the higher dimensional setting, in which case, we would simply have an estimate of the form $\iint \mathscr Z_m \dot u_m^2\geq -C_{\mathrm{sob}}\int_0^T\Vert \dot u_m(t,\cdot)\Vert_{W^{1,2}}\cdot \Vert \tilde{ \mathscr Z}_m(t,\cdot)\Vert_{L^p}\cdot \Vert \dot u_m(t,\cdot)\Vert_{L^2}dt$, where we would have used the three exponents' H\"{o}lder inequality with $1/p+1/q+1/2=1$, with $C_{\mathrm{sob}}$ the constant of the embedding $W^{1,2}\hookrightarrow L^q.$ 
\end{remark}
We obtain 
\begin{multline}
\ddot{\mathcal J}(m)[h,h]\geq  \alpha\iint_\TeT |{\partial_x} \dot u_m|^2-\mathfrak M(2)C_{\mathrm{sob}}\iint_\TT \dot u_m^2\\-\mathfrak M(2)C_{\mathrm{sob}}\int_0^T \Vert {\partial_x} \dot u_m(t,\cdot)\Vert_{L^2(\T)}\Vert \dot u_m(t,\cdot)\Vert_{L^2(\T)}dt-\gamma\int_\T \dot u^2(T,\cdot).\end{multline}
We perform one last step: from the arithmetic-geometric inequality, for any $\tilde\delta>0$, 
\begin{equation}\Vert {\partial_x}\dot u_m(t,\cdot)\Vert_{L^2(\T)}\cdot\Vert u_m(t,\cdot)\Vert_{L^2(\T)}
\leq \tilde\delta \Vert {\partial_x} \dot u_m(t,\cdot)\Vert_{L^2(\T)}^2+\frac1{\tilde\delta} \Vert  \dot u_m(t,\cdot)\Vert_{L^2(\T)}^2.\end{equation}
Fix $\delta>0$ and let $\tilde\delta:=\frac{\alpha}{\mathfrak M(2)C_{\mathrm{sob}}}\delta$. \corr{Thus, for some constant $\beta>0$} we have the estimate
\begin{multline}
\ddot{\mathcal J}(m)[h,h]\geq  \corr{\alpha(1-\delta)}\iint_\TeT |{\partial_x} \dot u_m|^2-\corr{\alpha\delta\iint_{(T-\e;T)\times \T}|\partial_x\dot u_m|^2}
\\-\corr{\beta\left(1+\frac1\delta\right)}\iint_\TT \dot u^2-\gamma\int_\T \dot u^2(T,\cdot).\end{multline}
The proof is now complete.

 \end{proof}
This proposition indicates that a possibility to derive a proof of theorem \ref{Th:Bgbg} is as follows: first, picking a maximiser $m^*$ of \eqref{Eq:Pv}, we argue by contradiction and assume $m^*$ is not bang-bang, so that the set $\omega=\{0<m^*<1\}$ has positive measure. Thus, for any admissible perturbation $h$ supported in $\omega$, $\dot{\mathcal J}(m^*)[h]=0$. If we can pick an admissible perturbation such that $\iint_\TeT|{\partial_x} \dot u_m|^2\gg \iint_\TT \dot u_m^2+\int_\T \dot u_m^2(T,\cdot)$, the second order Gateau derivative is positive, in contradiction with the optimality of $m^*$. To build such an $h$, we need to choose it highly oscillating; in other words, its Fourier series only has high order modes. Thus, the next sections are, respectively, devoted to the construction of an admissible $h$ that only has high Fourier modes, and to the study of the ensuing $\dot u_m$ via two-scale asymptotic expansions.

\emph{Throughout, we thus {argue by contradiction by considering} a non bang-bang maximiser $m^*$ and defining $\omega:=\{0<m^*<1\}$.}
\subsection{Construction of an admissible perturbation}\label{Se:Construction}
The relevant function to study is $\dot u_m$, which solves the parabolic equation 
$${\partial_t}\dot u_m -{\partial^2_{xx}} \dot u_m-V_m\dot u_m=hu_m.$${Recall that we argue by contradiction, and consider a maximiser $m^*$ that is not bang-bang, so that the set \[\omega=\{0<m^*<1\}\] has positive measure:
\[ \mathrm{Vol}(\omega)>0.\]}
We want to build $h$ such that, a large integer $K$ being fixed, $h$ is supported in $\omega$ and has the Fourier decomposition
$$h(x)=\sum_{k= K}^\infty a_k \cos(kx)+b_k\sin(kx).$$ 
Let us prove that such an admissible perturbation exists: let $\omega:=\{0<m^*<1\}$. As $m^*$ is not bang-bang, $\mathrm{Vol}(\omega)>0$. Consequently the space $L^2(\omega)$ is infinite dimensional. We identify each $H\in L^2(\omega)$ with $h:=H\mathds 1_\omega \in L^2(\T)$. We fix an integer $K\in \N\backslash\{0\}$ and we define, for any $0\leq k\leq K-1$ the linear functionals
$$T_k^1:L^2(\omega)\ni H \mapsto \int_\omega H{(x)}\cos(kx)dx\,, T_{{k}}^2:L^2(\omega)\ni H\mapsto \int_\omega H{(x)}\sin(kx)dx.$$
Finally, we define 
$$E_K:=\bigcap\limits_{k=0}^{K-1}\left( \ker\left(T_k^1\right)\cap \ker\left(T_k^2\right)\right)$$
{The linear subspace} $E_K$, as a finite intersection of closed hyperplanes, has finite co-dimension. It is, in particular, infinite dimensional. Hence, we can pick $H_K\in E_K$ such that $ \Vert H_K\Vert_{L^2(\omega)}>0$. By definition, $h_K:=H_k\mathds 1_\omega$ has the Fourier decomposition
$$h_K(x)=\sum_{k=K}^\infty a_k \cos(kx)+b_k\sin(kx)\text{ with, up to renormalisation, }\sum_{k=K}^\infty a_k^2+b_k^2=1.$$ 

We now want to study how $\dot u_m[h_K]$ behaves, for $K$ large. This prompts us to considering, first, the case of single cosines and sines.

\subsection{Computations for single-mode perturbations}
We emphasise once again that the computations of this paragraph are formal; we refer to proposition \ref{Pr:Rosecroix} for the rigorous proof of the expansions.
\paragraph{Case of single cosines}
 Let, for any $k\in \N\backslash\{0\}$, $\eta_k$ be the solution of 
 \begin{equation}\label{Eq:Etak}
 \begin{cases}
 \partial_t\eta_k-{\partial^2_{xx}} \eta_k-V_m\eta_k=u_m(t,x)\cos(kx)&\text{ in }\TT\,, 
 \\ \eta_k(0,\cdot)=0& \text{ in }\T.\end{cases}
 \end{equation}
 A natural expansion to look for is of the form
 $$ \eta_k(t,x)\approx \frac1{k^2}R_1(x,kx,t,k^2t)+\frac1{k^3}R_2(x,kx,t,k^2t)+\dots$$ By convention, we call $y$ and $s$ the second and fourth variables of $R_1$ and $R_2$. Plugging this formal expansion in \eqref{Eq:Etak} we obtain the following equations:
 \begin{equation}\label{Eq:R1}
 \begin{cases}
 \partial_sR_1-\partial^2_{yy}R_1=\cos(y)u_m(t,x)\,, 
 \\ R_1(x,y,0,0)=0,\end{cases}\end{equation}and
 \begin{equation}\label{Eq:R2}
 \begin{cases}
 \partial_sR_2-\partial^2_{yy}R_2=2\partial^2_{xy}R_1(x,y,t,s)\,, 
 \\ R_2(x,y,0,0)=0.\end{cases}\end{equation}

\eqref{Eq:R1} can be solved explicitly and we obtain 
\begin{equation}\label{Sol:R1}
R_1(x,y,t,s)=u_m(t,x) (1-e^{-s})\cos(y).
\end{equation}
This allows to derive the explicit form of \eqref{Eq:R2}. Namely, $R_2$ satisfies
\begin{equation}
\partial_s R_2-\partial^2_{yy}R_2=-2{\partial_x}u_m(t,x)\cdot(1-e^{-s})\sin(y).
\end{equation}
As a consequence, we look for $R_2$ under the form 
$$R_2(x,y,t,s)=-2{\partial_x}u_m(t,x)\cdot\sin(y)\varphi(s).$$ The function $\varphi$ satisfies
$$\varphi'+\varphi=1-e^{-s}$$ which can be integrated explicitly as 
$$\p(s)=1-se^{-s}-e^{-s}.$$
Finally, 
\begin{equation}\label{Sol:R2}
R_2(x,y,t,s)=-2{\partial_x}u_m(t,x)\cdot\sin(y)(1-se^{-s}-e^{-s}).
\end{equation}

\paragraph{Case of single sines} We then consider the case of single sines. Let, for any $k\in \N\backslash\{0\}$, $\zeta_k$ be the solution of
 \begin{equation}\label{Eq:Zetak}
 \begin{cases}
 \partial_t\zeta_k-{\partial^2_{xx}} \zeta_k-V\zeta_k=u_m(t,x)\sin(kx)&\text{ in }\TT\,, 
 \\ \zeta_k(0,\cdot)=0&\text{ in }\T.\end{cases}
 \end{equation}
Similarly, we look for an expansion in the form $$ \zeta_k(t,x)\approx \frac1{k^2}S_1(x,kx,t,k^2t)+\frac1{k^3}S_2(x,kx,t,k^2t)+\dots$$ By convention, we call $y$ and $s$ the second and fourth variables of $S_1$ and $S_2$. Plugging this formal expansion in \eqref{Eq:Zetak} we obtain the following equations:
 \begin{equation}\label{Eq:S1}
 \begin{cases}
 \partial_sS_1-\partial^2_{yy}S_1=\sin(y)u_m(t,x)\,, 
 \\ S_1(x,y,0,0)=0,\end{cases}\end{equation}and
 \begin{equation}\label{Eq:S2}
 \begin{cases}
 \partial_sS_2-\partial^2_{yy}S_2=2\partial^2_{xy}S_1(x,y,t,s)\,, 
 \\ S_2(x,y,0,0)=0.\end{cases}\end{equation}

\eqref{Eq:S1} can be solved explicitly and we have 
\begin{equation}\label{Sol:S1}
S_1(x,y,t,s)=u_m(t,x) (1-e^{-s})\sin(y).
\end{equation}
This allows to derive the explicit form of \eqref{Eq:S2}. Namely, $S_2$ satisfies
\begin{equation}
\partial_s S_2-\partial^2_{yy}S_2=2{\partial_x}u_m(t,x)\cdot(1-e^{-s})\cos(y).
\end{equation} Proceeding as in the computations of $R_2$ we derive
\begin{equation}\label{Sol:S2}
S_2(x,y,t,s)=2{\partial_x}u_m(t,x)\cdot\cos(y)(1-se^{-s}-e^{-s}).
\end{equation}

Of course we wish to write an approximation of the type
\begin{multline}\label{Ilias} \dot u_m\approx \sum_{k=K}^\infty a_k\left(\frac1{k^2}R_1(x,kx,t,k^2t)+\frac1{k^3}R_2(x,kx,t,k^2t)\right)\\+\sum_{k=K}^\infty b_k\left(\frac1{k^2}S_1(x,kx,t,k^2t)+\frac1{k^3}S_2(x,kx,t,k^2t)\right).\end{multline}
In order to determine how strong this approximation should be to yield an exploitable result on $\ddot{\mathcal J}$, we next study the leading term of \eqref{Eq:EstDDJ}, should the expansion \eqref{Ilias} hold.

\subsection{Formal estimate of the leading order term}
We work under the assumption that
\begin{align*}
\dot u_m&\approx\sum_{k=K}^\infty a_k \eta_k+\sum_{k=K}^\infty b_k \zeta_k
\\&\approx  \sum_{k=K}^\infty a_k\left( u_m(t,x)\frac1{k^2}\cos(kx)(1-e^{-k^2t})-{\partial_x}u_m\cdot\frac2{k^3}\sin(kx)(1-k^2te^{-k^2t}-e^{-k^2t})\right)
\\&+\sum_{k=K}^\infty b_k\left( u_m(t,x)\frac1{k^2}\sin(kx)(1-e^{-k^2t})+{\partial_x}u_m\cdot\frac2{k^3}\cos(kx)(1-k^2te^{-k^2t}-e^{-k^2t})\right)
\end{align*}
In particular, we have (this is still formal, at this point)
\begin{align*}
\partial_x\dot u_m&\approx\sum_{k=K}^\infty a_k\left({\partial_x}u_m\cdot \frac1{k^2}\cos(kx)(1-e^{-k^2t})-{\partial^2_{xx}}u_m\cdot\frac2{k^3}\sin(kx)(1-k^2te^{-k^2t}-e^{-k^2t})\right)
\\&+\sum_{k=K}^\infty a_k\left(- u_m\frac1{k}\sin(kx)(1-e^{-k^2t})-{\partial_x}u_m\cdot\frac2{k^2}\cos(kx)(1-k^2te^{-k^2t}-e^{-k^2t})\right)
\\&+\sum_{k=K}^\infty b_k\left({\partial_x}u_m\cdot \frac1{k^2}\sin(kx)(1-e^{-k^2t})+{\partial^2_{xx}}u_m\cdot \frac2{k^3}\cos(kx)(1-k^2te^{-k^2t}-e^{-k^2t})\right)
\\&+\sum_{k=K}^\infty b_k\left( u_m\frac1{k}\cos(kx)(1-e^{-k^2t})-{\partial_x}u_m\cdot\frac2{k^2}\sin(kx)(1-k^2te^{-k^2t}-e^{-k^2t})\right)
\\&=u_m\left\{ \sum_{k=K}^\infty\frac{- a_k\sin(kx)+b_k\cos(kx)}{k}\left(1-e^{-k^2t}\right)\right\}&\left(=:L_K\right)
\\&+{\partial_x}u_m\cdot\left\{\sum_{k=K}^\infty \frac{a_k\cos(kx)-b_k\sin(kx)}{k^2}\left(-1+e^{-k^2t}+2k^2te^{-k^2t}\right)\right\}&\left(=:I_K\right)
\\&-2{\partial^2_{xx}}u_m\cdot\left\{
\sum_{k=K}^\infty\frac{ a_k\sin(kx)-b_k\cos(kx)}{k^3}\left(1-k^2te^{-k^2t}-e^{-k^2t}\right)\right\}&\left(=:J_K\right).
\end{align*}
Thus, we should have 
\begin{equation}\label{Eq:Dec}\iint_\TeT |{\partial_x} \dot u_m|^2\approx\iint_\TeT\left(I_K+J_K+L_K\right)^2.\end{equation} Given the expressions for $I_K\,, J_K\,, L_K$, we expect $\iint_\TeT L_K^2$ to be leading in \eqref{Eq:Dec}. For this reason we first bound the right-hand side of \eqref{Eq:Dec} from below; we shall use the inequality
\begin{equation}\label{Ineq1}|xy|\leq \e x^2+\frac1\e y^2\text{ for any $x\,, y \in \R\,, \e>0$}\end{equation}as well as
\begin{equation}\label{Ineq2}(x+y)^2\leq 2\left(x^2+y^2\right)\text{ for any $x,y\in \R.$}\end{equation} We obtain
\begin{align}
\iint_\TeT \left(I_K+J_K+L_K\right)^2&=\iint_\TeT L_K^2+2\iint_\TeT L_K(I_K+J_K)\\&+\iint_\TeT (I_K+J_K)^2
\\&\geq \iint_\TeT L_K^2-\frac12\iint_\TeT L_K^2
\\&-4\iint_\TeT \left(I_K+J_K\right)^2 +\iint_\TeT \left(I_K+J_K\right)^2\\&\text{ from \eqref{Ineq1} with $\e=\frac12$}
\\\label{Sol}&=\frac12\iint_\TeT L_K^2-6\iint_\TeT I_K^2-6\iint_\TeT J_K^2.
\end{align}
We recall that from lemma \ref{Le:ubdd} we have 
$$\underline{d}:=\inf_\TT u_m>0.$$This will prove crucial.

We shall estimate the three terms (\emph{i.e. } $\iint L_K^2\,, \iint I_K^2\,, \iint J_K^2$) separately.

\paragraph{Estimate of $L_K$}
We have 
\begin{align*}
\iint_\TeT L_K^2&=\iint_\TeT u_m^2(t,x)\left\{ \sum_{k=K}^\infty\frac{- a_k\sin(kx)+b_k\cos(kx)}{k}\left(1-e^{-k^2t}\right)\right\}^2dxdt
\\&\geq \underline{d}^2\iint_\TeT \left\{ \sum_{k=K}^\infty\frac{- a_k\sin(kx)+b_k\cos(kx)}{k}\left(1-e^{-k^2t}\right)\right\}^2dxdt
\\&=\underline{d}^2\pi\sum_{k=K}^\infty\left(a_k^2+b_k^2\right)\int_0^{T-\e}\frac1{k^2} \left( 1-e^{-k^2t}\right)^2dt
\\&\geq 2C_0 \sum_{k=K}^\infty\frac{a_k^2+b_k^2}{k^2}
\end{align*}
for a positive constant $C_0>0$. We have hence obtained
\begin{equation}\label{Est:LK}\iint_\TeT L_K^2\geq 2 C_0 \sum_{k=K}^\infty\frac{a_k^2+b_k^2}{k^2}\text{ for a constant $C_0>0$.}\end{equation}

\begin{remark}We choose $2C_0$ to obtain a cleaner estimate on the second order derivative.\end{remark}

\paragraph{Estimate of $I_K$} We first recall that from proposition \ref{Pr:Regularity}
\[\overline{d}:=\sup_{t\in[0,T]}\Vert u_m\Vert_{\mathscr C^1(\T)}<\infty.\]
We then bound and compute 
\begin{align*}
\iint_\TeT I_K^2&=4\iint_\TeT \left({\partial_x}u_m\right)^2\left\{\sum_{k=K}^\infty \frac{a_k\cos(kx)-b_k\sin(kx)}{k^2}\left(-1+e^{-k^2t}+2k^2te^{-k^2t}\right)\right\}^2dtdx
\\&\leq 4{\overline{d}}^2\iint_\TeT\left\{\sum_{k=K}^\infty \frac{a_k\cos(kx)-b_k\sin(kx)}{k^2}\left(-1+e^{-k^2t}+2k^2te^{-k^2t}\right)\right\}^2dxdt
\\&=4\pi\overline{d}^2\sum_{k=K}^\infty\frac{a_k^2+b_k^2}{k^4}\int_0^{T-\e}\left(-1+e^{-k^2t}(2k^2t+1)\right)^2dt
\\&=4\pi{\overline{d}}^2\sum_{k=K}^\infty\frac{a_k^2+b_k^2}{k^4}\int_0^{T-\e}(1+4e^{-2k^2t}k^4t^2-2e^{-k^2t}-4k^2te^{-k^2t}+4k^2t e^{-2k^2t}+e^{-2k^2t})dt.
\end{align*}
However, each of the integrals can be computed explicitly:
{Indeed, first observe that:}
\begin{equation*}
\int_0^{T-\e} t^2e^{-2k^2t}dt=-\frac{(T-\e)^2e^{-2k^2(T-\e)}}{2k^2}-\frac{(T-\e)e^{-2k^2(T-\e)}}{2k^4}+\frac1{4k^6}\left(1-e^{-2k^2(T-\e)}\right)
\leq \frac{N_{0,I}}{k^6}\end{equation*}
{ for some constant $N_{0,I}$.}

{Second, we  obtain in a similar fashion the estimate:}
\begin{equation*} 
 \int_0^{T-\e} e^{-2k^2t}dt\leq \frac{N_{1,I}}{k^2}\end{equation*}{ for some constant $N_{1,I}$}.
 
 {Finally, the same type of computations leads to}\begin{equation*}
\int_0^{T-\e} te^{-k^2t}dt\leq \frac{N_{2,I}}{k^4}\end{equation*}{ for some constant $N_{2,I}$}.
Hence, there exists $K\in \N$ and a  constant $N_{3,I}$ such that for any $k\geq K$,
\begin{equation}\label{Chapoutot}
\int_0^{T-\e} (1+4e^{-2k^2t}k^4t^2-2e^{-2k^2t}-4k^2te^{-k^2t}+4k^2t e^{-2k^2t}+e^{-2k^2t})dt\leq {N_{3,I}}\left(1+\frac1{k^2}\right).\end{equation}
Consequently, there exists a positive constant $C_1$ such that
\begin{equation}\label{Est:IK}
\iint_\TeT I_K^2\leq \frac{C_1}6\sum_{k=K}^\infty\frac{a_k^2+b_k^2}{k^4}.\end{equation}

\paragraph{Estimate on $J_K$}
This last term is the trickiest one. Indeed, we do not have ${\partial^2_{xx}}u_m\in L^\infty$, but simply, from proposition \ref{Pr:Regularity}, 
$$\forall p\in [1;+\infty)\,, \sup_{t\in[0,T]}\left\Vert {\partial^2_{xx}}u_m(t,\cdot)\right\Vert_{L^p(\T)}=:\mathfrak M(p)<\infty.$$
However, we can use the same trick as in bounding the second order derivative (see the proof of proposition \ref{Pr:Cruise}). We indeed obtain
\begin{align*}
\iint_\TeT J_K^2&=\iint_\TeT 4\left({\partial^2_{xx}}u_m\right)^2\left\{\underbrace{
\sum_{k=K}^\infty\frac{ a_k\sin(kx)-b_k\cos(kx)}{k^3}\left(1-k^2te^{-k^2t}-e^{-k^2t}\right)}_{=:W_K}\right\}^2
\\&\leq \underbrace{16 \mathfrak M(4)^2C_{sob}}_{=:D_{0,J}}\int_0^{T-\e} \Vert W_K(t,\cdot)\Vert_{L^2(\T)}\Vert  W_K(t,\cdot)\Vert_{W^{1,2}(\T)}dt
\\&\leq D_{0,J}\iint_\TeT W_K^2+D_{0,J}\int_0^{T-\e}\Vert W_K(t,\cdot)\Vert_{L^2(\T)}\Vert {\partial_x} W_K(t,\cdot)\Vert_{L^2(\T)}dt,
\end{align*}
where $C_{sob}$ is the constant of the (one-dimensional) embedding $W^{1,2}(\T)\hookrightarrow \mathscr C^0(\T)$.

We compute, for every $t\in [0,T]$, both $\Vert W_K(t,\cdot)\Vert_{L^2(\T)}$ and $\Vert \n W_K(t,\cdot)\Vert_{L^2(\T)}$.
First,
\begin{align*}
\int_\T W_K^2(t,\cdot)&=\int_\T \left\{
\sum_{k=K}^\infty\frac{ a_k\sin(kx)-b_k\cos(kx)}{k^3}\left(1-k^2te^{-k^2t}-e^{-k^2t}\right)\right\}^2dx
\\&=\pi\sum_{k=K}^\infty\frac{a_k^2+b_k^2}{k^6}\left(1-k^2te^{-k^2t}-e^{-k^2t}\right)^2.
\end{align*}
Second,
\begin{align*}
\int_\T \left|{\partial_x} W_K\right|^2(t,\cdot)&=\int_\T \left\{
\sum_{k=K}^\infty\frac{- a_k\sin(kx)-b_k\cos(kx)}{k^2}\left(1-k^2te^{-k^2t}-e^{-k^2t}\right)\right\}^2dx
\\&=\pi\sum_{k=K}^\infty\frac{a_k^2+b_k^2}{k^4}\left(1-k^2te^{-k^2t}-e^{-k^2t}\right)^2.
\end{align*}
We notice that
\begin{equation}
\int_\T \left|{\partial_x} W_K\right|^2(t,\cdot)\geq K^2\int_\T W_K^2(t,\cdot)\end{equation} or, in other terms, that, for any $t\in [0,T]$,
\begin{equation}
\Vert W_K(t,\cdot)\Vert_{L^2(\T)}\leq\frac1K\Vert {\partial_x}W_K(t,\cdot)\Vert_{L^2(\T)}.\end{equation}
Consequently,
\begin{multline*}
\iint_\TeT W_K^2+\int_0^{T-\e}\Vert W_K(t,\cdot)\Vert_{L^2(\T)}\Vert {\partial_x} W_K(t,\cdot)\Vert_{L^2(\T)}dt\\\leq \iint_\TeT W_K^2+\int_0^{T-\e}\Vert W_K(t,\cdot)\Vert_{L^2(\T)}\Vert {\partial_x} W_K(t,\cdot)\Vert_{L^2(\T)}dt
\\\leq \left(\frac1{K^2}+\frac1K\right)\int_0^{T-\e} ||{\partial_x} W_K||_{L^2(\T)}^2(t,\cdot)
\\\leq \frac{2\pi}K \sum_{k=K}^\infty \frac{a_k^2+b_k^2}{k^4}\int_0^T(1-e^{-k^2t}(k^2t+1))^2dt.
 \end{multline*}
 From the same computations that established \eqref{Chapoutot}, there exists a constant $D_{1,J}$ such that, whenever $K$ is large enough, for any $k\geq K$, there holds
 \begin{equation}\label{Ingrao}
\int_0^T(1-e^{-k^2t}(k^2t+1))^2dt\leq {D_{1,J}},
 \end{equation}
and so, finally, for a constant $C_2$,\begin{equation}\label{Est:JK}
\iint_\TeT J_K^2\leq \frac{C_2}{6}\sum_{k=K}^\infty \frac{a_k^2+b_k^2}{k^4}.
\end{equation}

Combining \eqref{Sol}-\eqref{Est:LK}-\eqref{Est:IK}-\eqref{Est:JK} we finally derive the following lower-bound on the leading term: there exists $C_{\mathrm{lead}}>0$, independent of $K$ such that
\begin{equation}\label{Est:Defin}
\iint_\TeT\left(I_K+J_K+L_K\right)^2\geq C_{\mathrm{lead}}\sum_{k=K}^\infty \frac{a_k^2+b_k^2}{k^2}.\end{equation}

\corr{Proceeding in exactly the same fashion, we obtain the existence of $C_{\mathrm{lead}}'$, independent of $K$ and of $\e>0$, such that 
\begin{equation}\label{Eq:EstDefin2} \iint_{(T-\e;T)\times \T}(I_K+J_K+L_K)^2\leq C_{\mathrm{lead}}'\sum_{k=K}^\infty \frac{a_k^2+b_k^2}{k^2}.\end{equation}
}

\subsection{Formal estimate of the lower order term}
If we assume that \begin{align*}
\dot u_m&\approx\sum_{k=K}^\infty a_k \eta_k+\sum_{k=K}^\infty b_k \zeta_k
\\&=  \sum_{k=K}^\infty a_k\left( u(t,x)\frac1{k^2}\cos(kx)(1-e^{-k^2t})-{\partial_x}u_m\cdot\frac2{k^3}\sin(kx)(1-k^2te^{-k^2t}-e^{-k^2t})\right)
\\&+\sum_{k=K}^\infty b_k\left( u(t,x)\frac1{k^2}\sin(kx)(1-e^{-k^2t})+{\partial_x}u_m\cdot\frac2{k^3}\cos(kx)(1-k^2te^{-k^2t}-e^{-k^2t})\right),
\end{align*}
then, in the very same way, we obtain the existence of a constant $C_{\mathrm{low}}$ such that
\begin{equation}\label{Est:Defin2}
\iint_\TT \left\{\sum_{k=K}^\infty a_k\eta_k+\beta_k\zeta_k\right\}^2\leq C_{\mathrm{low}}\sum_{k=K}^\infty \frac{a_k^2+b_k^2}{k^4}.\end{equation}

\subsection{Strategy and comment for the proof of the asymptotic expansion}
We shall now establish rigorously a strong enough approximation result.
Let us define 
\begin{multline*}Z_K:=\sum_{k=K}^\infty a_k \left( \frac{u_m}{k^2}\cos(kx)(1-e^{-k^2t})-\frac{2{\partial_x u_m}}{k^3}\sin(kx)(1-k^2te^{-k^2t}-e^{-k^2t})\right)
\\+\sum_{k=K}^\infty b_k \left( \frac{u_m}{k^2}\sin(kx)(1-e^{-k^2t})+\frac{2\partial_x u_m}{k^3}\cos(kx)(1-k^2te^{-k^2t}-e^{-k^2t})\right).
\end{multline*}

From \eqref{Est:Defin2}-\eqref{Est:Defin}, we need the following proposition to prove the theorem (see also lemma \ref{Le:Rosecroixsuffices} below, which proves that this proposition is enough):

\begin{proposition}\label{Pr:Rosecroix}There exists a constant $C_{\mathrm{cont}}$ such that, for any $\Upsilon>0$, 
\begin{multline}
\int_0^T \left\Vert \dot u_m-Z_K(t,\cdot)\right\Vert_{W^{1,2}(\T)}^2dt+\Vert \dot u_m(T,\cdot)-Z_K(T,\cdot)\Vert_{L^2(\T)}^2\leq \frac{C_{\mathrm{cont}}}\Upsilon\sum_{k=K}^\infty\frac{a_k^2+b_k^2}{k^4}\\+\Upsilon C_{\mathrm{cont}} \sum_{k=K}^\infty\frac{a_k^2+b_k^2}{k^2}.
\end{multline}
\end{proposition}
The object of the next lemma is to prove that proposition \ref{Pr:Rosecroix} suffices to obtain theorem \ref{Th:Bgbg}.

\begin{lemma}\label{Le:Rosecroixsuffices}
Proposition \ref{Pr:Rosecroix} implies theorem \ref{Th:Bgbg}.\end{lemma}
\begin{proof}[Proof of Lemma \ref{Le:Rosecroixsuffices}]
We use proposition \ref{Pr:Cruise} with the perturbation $h_K$ constructed in Section \ref{Se:Construction}. We study the right-hand side of \eqref{Eq:EstDDJ}. On the one-hand, we have \begin{align*}
\iint_\TeT|{\partial_x} \dot u_m|^2&\geq\frac12 \iint_\TeT |{\partial_x} Z_K|^2-\iint_\TT\left|{\partial_x} \dot u_m-{\partial_x} Z_K\right|^2dtdx
\\&\geq \frac{C_{\mathrm{lead}}}{2}\sum_{k=K}^\infty \frac{a_k^2+b_k^2}{k^2}-\frac{C_{\mathrm{cont}}}{\Upsilon} \sum_{k=K}^\infty\frac{a_k^2+b_k^2}{k^4}-{\Upsilon C_{\mathrm{cont}}}\sum_{k=K}^\infty\frac{a_k^2+b_k^2}{k^2}
\end{align*}
and 
\corr{
\begin{align*}
\iint_{(T-\e;T)\times \T}|{\partial_x} \dot u_m|^2&\leq2 \iint_{(T-\e;T)} |{\partial_x} Z_K|^2+2\iint_\TT\left|{\partial_x} \dot u_m-{\partial_x} Z_K\right|^2dtdx
\\&\leq 2{C_{\mathrm{lead}}'}{}\sum_{k=K}^\infty \frac{a_k^2+b_k^2}{k^2}+2\frac{C_{\mathrm{cont}}}{\Upsilon} \sum_{k=K}^\infty\frac{a_k^2+b_k^2}{k^4}+2{\Upsilon C_{\mathrm{cont}}}\sum_{k=K}^\infty\frac{a_k^2+b_k^2}{k^2}
\end{align*}
}

On the other hand, 
\begin{align*}
\iint_\TT \dot u_m^2&\leq  2 \iint_\TT | Z_K|^2+2\iint_\TT\left| \dot u_m- Z_K\right|^2dtdx
\\&\leq 2{C_{\mathrm{low}}}{}\sum_{k=K}^\infty \frac{a_k^2+b_k^2}{k^4}+\frac{2C_{\mathrm{cont}}}{\Upsilon} \sum_{k=K}^\infty\frac{a_k^2+b_k^2}{k^4}+2\Upsilon{C_{\mathrm{cont}}}\sum_{k=K}^\infty\frac{a_k^2+b_k^2}{k^2}
\end{align*}
for $K$ large enough, from \eqref{Est:Defin2},  and Proposition \ref{Pr:Rosecroix},
and, in the same way,
\begin{align*}
\int_\T \dot u_m^2(T,\cdot)&\leq \frac{2C_{\mathrm{cont}}}{\Upsilon} \sum_{k=K}^\infty\frac{a_k^2+b_k^2}{k^4}+2{\Upsilon C_{\mathrm{cont}}}\sum_{k=K}^\infty\frac{a_k^2+b_k^2}{k^2}\end{align*}
Consequently, for any $\delta>0$ \corr{
\begin{align*}
\ddot{\mathcal J}(m^*)[h_K,h_K]&\geq\frac{\alpha(1-\delta) C_{\mathrm{lead}}-4\alpha\delta C_{\mathrm{lead}}'}{2}\sum_{k=K}^\infty\frac{a_k^2+b_k^2}{k^2}
+\frac{(-\alpha(1-\delta)-2\alpha\delta) C_{\mathrm{cont}}}{\Upsilon} \sum_{k=K}^\infty\frac{a_k^2+b_k^2}{k^4}
\\&{+(-\alpha(1-\delta)-2\alpha\delta) \Upsilon C_{\mathrm{cont}}}\sum_{k=K}^\infty\frac{a_k^2+b_k^2}{k^2}
-2\beta\left(1+\frac1\delta\right){C_{\mathrm{low}}}{}\sum_{k=K}^\infty \frac{a_k^2+b_k^2}{k^4}
\\&-2\frac{\beta\left(1+\frac1\delta\right) C_{\mathrm{cont}}}{\Upsilon} \sum_{k=K}^\infty\frac{a_k^2+b_k^2}{k^4}-2\beta\left(1+\frac1\delta\right)\Upsilon{C_{\mathrm{cont}}}\sum_{k=K}^\infty\frac{a_k^2+b_k^2}{k^2}
\\&-2\gamma{C_{\mathrm{low}}}{}\sum_{k=K}^\infty \frac{a_k^2+b_k^2}{k^4}-2\gamma\frac{C_{\mathrm{cont}}}{\Upsilon} \sum_{k=K}^\infty\frac{a_k^2+b_k^2}{k^4}-2\gamma\Upsilon{C_{\mathrm{cont}}}\sum_{k=K}^\infty\frac{a_k^2+b_k^2}{k^2}
\\&=\left(\sum_{k=K}^\infty\frac{a_k^2+b_k^2}{k^2}\right)
\left\{\frac{\alpha(1-\delta) C_{\mathrm{lead}}-4\alpha\delta C_{\mathrm{lead}}'}{2}-\alpha(1+\delta) \Upsilon C_{\mathrm{cont}}-2\beta\left(1+\frac1\delta\right) \Upsilon C_{\mathrm{cont}}-2\gamma \Upsilon C_{\mathrm{cont}}\right\}
\\&-\left(\sum_{k=K}^\infty\frac{a_k^2+b_k^2}{k^4}\right)\left\{\frac{\alpha(1+\delta) C_{\mathrm{cont}}}\Upsilon-2\beta\left(1+\frac1\delta\right) C_{\mathrm{low}}-2\frac{\beta\left(1+\frac1\delta\right) C_{\mathrm{cont}}}\Upsilon-2\gamma  C_{\mathrm{low}}-2\gamma \Upsilon  C_{\mathrm{cont}}\right\}
\\&\geq \left(\sum_{k=K}^\infty\frac{a_k^2+b_k^2}{k^2}\right)
\left\{\frac{\alpha(1-\delta) C_{\mathrm{lead}}-4\alpha\delta C_{\mathrm{lead}}'}2-\alpha(1+\delta) \Upsilon C_{\mathrm{cont}}-2\beta\left(1+\frac1\delta\right) \Upsilon C_{\mathrm{cont}}-2\gamma \Upsilon C_{\mathrm{cont}}\right\}
\\&-\left(\sum_{k=K}^\infty\frac{a_k^2+b_k^2}{k^2}\right)\frac{\left|\frac{\alpha(1+\delta) C_{\mathrm{cont}}}\Upsilon-2\beta\left(1+\frac1\delta\right) C_{\mathrm{low}}-2\frac{\beta\left(1+\frac1\delta\right) C_{\mathrm{cont}}}\Upsilon-2\gamma  C_{\mathrm{low}}-2\gamma \Upsilon  C_{\mathrm{cont}}\right|}{K^2}
\end{align*}
We first pick $\delta>0$ small enough to ensure that 
\[\alpha(1-\delta)C_{\mathrm{lead}}-4\alpha\delta C_{\mathrm{lead}}'>0.\]
}{
We then choose $\Upsilon>0$ small enough so that
$$\alpha'':=\frac{\alpha(1-\delta) C_{\mathrm{lead}}-4\delta C_{\mathrm{lead}}'}2-\alpha(1+\delta) \Upsilon C_{\mathrm{cont}}-2\beta\left(1+\frac1\delta\right) \Upsilon C_{\mathrm{cont}}-2\gamma \Upsilon C_{\mathrm{cont}}>0.$$ We define 
\[\beta'':=\left|\frac{\alpha(1+\delta) C_{\mathrm{cont}}}\Upsilon-2\beta\left(1+\frac1\delta\right) C_{\mathrm{low}}-2\frac{\beta\left(1+\frac1\delta\right) C_{\mathrm{cont}}}\Upsilon-2\gamma  C_{\mathrm{low}}-2\gamma \Upsilon  C_{\mathrm{cont}}\right|.\]}
Thus we have the lower bound
\begin{equation}
\ddot{\mathcal J}(m^*)[h_K,h_K]\geq \left\{\alpha''-\frac{\beta''}{K^2}\right\} \left(\sum_{k=K}^\infty\frac{a_k^2+b_k^2}{k^2}\right).\end{equation} We pick $K$ large enough to ensure that 
\[\alpha''-\frac{\beta''}{K^2}\geq \frac{\alpha''}{2}\] and  it follows that 
\[\ddot{\mathcal J}(m^*)[h_K,h_K]>0,\] in contradiction with the optimality of $m^*$. The conclusion of the Theorem follows: every maximiser must be a bang-bang function.

\end{proof}
The rest of this section is devoted to the proof of proposition \ref{Pr:Rosecroix}.
\subsection{Proof of proposition \ref{Pr:Rosecroix}}
\begin{proof}[Proof of proposition \ref{Pr:Rosecroix}]
We recall that, in its expanded form, $Z_K$ writes
\begin{multline*}Z_K:=\sum_{k=K}^\infty a_k \left( \frac{u_m}{k^2}\cos(kx)(1-e^{-k^2t})-\frac{2{\partial_x u_m}}{k^3}\sin(kx)(1-k^2te^{-k^2t}-e^{-k^2t})\right)
\\+\sum_{k=K}^\infty b_k \left( \frac{u_m}{k^2}\sin(kx)(1-e^{-k^2t})+\frac{2\partial_x u_m}{k^3}\cos(kx)(1-k^2te^{-k^2t}-e^{-k^2t})\right)
\end{multline*}
We define the remainder term\[R_K:=\dot u_m-Z_K.\]

The computations needed in order to determine an explicit equation for $R_K$ are rather lengthy. We split them up.

Define $$T_K^1:=u_m\sum_{k=K}^\infty\frac{a_k\cos(kx)+b_k\sin(kx)}{k^2}(1-e^{-k^2t}).$$ We first have 
\begin{multline*}
{\partial_t}T_K^1={\partial_t}u_m\cdot \sum_{k=K}^\infty \frac{a_k\cos(kx)+b_k\sin(kx)}{k^2}(1-e^{-k^2t})+u_m\sum_{k=K}^\infty \left(a_k\cos(kx)+b_k\sin(kx)\right)e^{-k^2t}.\end{multline*}
Second, we have 
\begin{equation*}
{\partial_x}T_K^1={\partial_x}u_m\cdot \sum_{k=K}^\infty \frac{a_k\cos(kx)+b_k\sin(kx)}{k^2}(1-e^{-k^2t})-u_m\sum_{k=K}^\infty \frac{a_k\sin(kx)-b_k\cos(kx)}k (1-e^{-k^2t})
\end{equation*}
so that 
\begin{multline*}
{\partial^2_{xx}}T_K^1={\partial^2_{xx}}u_m\cdot\sum_{k=K}^\infty \frac{a_k\cos(kx)+b_k\sin(kx)}{k^2}(1-e^{-k^2t})-u_m\sum_{k=K}^\infty \left(a_k \cos(kx)+b_k\sin(kx)\right)(1-e^{-k^2t})\\-2{\partial_x}u_m\cdot\sum_{k=K}^\infty\frac{a_k\sin(kx)-b_k\cos(kx)}{k} (1-e^{-k^2t}).
\end{multline*}
Hence, introducing the  differential operator 
\[\mathcal L_{V_m}:\Phi\mapsto \partial_t \Phi-{\partial^2_{xx}} \Phi-V_m\Phi\] we obtain 
\begin{multline}\label{Eq:Premi}
\mathcal L_{V_m} T_K^1=\left(\mathcal L_{V_m} u_m\right)\sum_{k=K}^\infty \frac{a_k\cos(kx)+b_k\sin(kx)}{k^2}(1-e^{-k^2t})+u_m\sum_{k=K}^\infty (a_k \cos(kx)+b_k\sin(kx))\\+2{\partial_x}u_m\cdot\sum_{k=K}^\infty \frac{a_k\sin(kx)-b_k\cos(kx)}{k}\left(1-e^{-k^2t}\right)
\end{multline}

Second, we set
$$T_K^2:=-{2{\partial_x}u_m}\cdot\sum_{k=K}^\infty \frac{a_k\sin(kx)-b_k\cos(kx)}{k^3}(1-k^2te^{-k^2t}-e^{-k^2t}).$$ We obtain 
\begin{multline*}{\partial_t}T_K^2=-2{\partial_x} \left({\partial_t}u_m\right)\cdot\sum_{k=K}^\infty \frac{a_k\sin(kx)-b_k\cos(kx)}{k^3}(1-k^2te^{-k^2t}-e^{-k^2t})
\\-2{\partial_x}u_m\cdot\sum_{k=K}^\infty \left(a_k \sin(kx)-b_k\cos(kx)\right)kte^{-k^2t} .\end{multline*}  Let us define, in order to alleviate the upcoming computations, 
$$\p(s):=1-se^{-s}-e^{-s}.$$
Similarly we obtain 
\begin{equation*}
{\partial_x}T_K^2=-2{\partial^2_{xx}}u_m\cdot\sum_{k=K}^\infty\frac{a_k\sin(kx)-b_k\cos(kx)}{k^3}\p(k^2t)-2{\partial_x}u_m\cdot\sum_{k=K}^\infty \frac{a_k\cos(kx)+b_k\sin(kx)}{k^2} \p(k^2t)
\end{equation*}
as well as (the next equation should be understood in the $W^{-1,2}(\O)$ sense)
\begin{multline*}
{\partial^2_{xx}}T_K^2=-2{\partial_x}\left({\partial^2_{xx}}u_m\right)\cdot \sum_{k=K}^\infty\frac{a_k\sin(kx)-b_k\cos(kx)}{k^3}\p(k^2t)
\\-4{\partial^2_{xx}}u_m\cdot\sum_{k=K}^\infty \frac{a_k\cos(kx)+b_k\sin(kx)}{k^2} \p(k^2t)
\\+2{\partial_x}u_m\cdot\sum_{k=K}^\infty \frac{a_k\sin(kx)+b_k\cos(kx)}{k}\p(k^2t).
\end{multline*}

Combining these bricks we are left with 
\begin{align*}
\mathcal L_{V_m}T_K^2&=-2{\partial_x}\left(\mathcal L_{V_m} u_m\right)\cdot\sum_{k=K}^\infty \frac{a_k\sin(kx)-b_k\cos(kx)}{k^3}\p(k^2t)
\\&-2{\partial_x}u_m\cdot\sum_{k=K}^\infty \left(a_k \sin(kx)-b_k\cos(kx)\right)\left(kte^{-k^2t} +\frac{\p(k^2t)}k\right)
\\&-4{\partial^2_{xx}}u_m\cdot \sum_{k=K}^\infty \frac{a_k\cos(kx)+b_k\sin(kx)}{k^2} \p(k^2t)
\\&=-2{\partial_x}\left(\mathcal L_{V_m} u_m\right)\cdot\sum_{k=K}^\infty \frac{a_k\sin(kx)-b_k\cos(kx)}{k^3}\p(k^2t)
\\&-2{\partial_x}u_m\cdot\sum_{k=K}^\infty \frac{a_k \sin(kx)-b_k\cos(kx)}k\left(1-e^{-k^2t}\right)
\\&-4{\partial^2_{xx}}u_m\cdot\sum_{k=K}^\infty \frac{a_k\cos(kx)+b_k\sin(kx)}{k^2} \p(k^2t),
\end{align*}
where, to obtain the last equality, we simply wrote
\[kte^{-k^2t}+\frac{\p(k^2t)}k=kte^{-k^2t}+\frac1k-kte^{-k^2t}-\frac{e^{-k^2t}}k=\frac{1-e^{-k^2t}}k.
\]

Now, it follows that $R_K$ satisfies
\begin{align*}
\mathcal L_{V_m} R_K&=\mathcal L_{V_m} \dot u_m-\mathcal L_{V_m} T_K^1-\mathcal L_{V_m} T_K^2
\\&=u_m\sum_{k=K}^\infty a_k \cos(kx)+b_k\sin(kx)
-\left(\mathcal L_{V_m} u\right)\sum_{k=K}^\infty \frac{a_k\cos(kx)+b_k\sin(kx)}{k^2}(1-e^{-k^2t})
\\&-u_m\sum_{k=K}^\infty (a_k \cos(kx)+b_k\sin(kx))
\\&-2{\partial_x}u_m\cdot\sum_{k=K}^\infty \frac{a_k\sin(kx)-b_k\cos(kx)}{k}\left(1-e^{-k^2t}\right)
\\&+2{\partial_x}\left(\mathcal L_{V_m} u_m\right)\cdot\sum_{k=K}^\infty \frac{a_k\sin(kx)-b_k\cos(kx)}{k^3}\p(k^2t)
\\&+2{\partial_x}u_m\cdot\sum_{k=K}^\infty \frac{a_k \sin(kx)-b_k\cos(kx)}k\left(1-e^{-k^2t}\right)
\\&+4{\partial^2_{xx}}u_m\cdot\sum_{k=K}^\infty \frac{a_k\cos(kx)+b_k\sin(kx)}{k^2} \p(k^2t)
\\&=-\left(\mathcal L_{V_m} u_m\right)\sum_{k=K}^\infty \frac{a_k\cos(kx)+b_k\sin(kx)}{k^2}(1-e^{-k^2t})
\\&+2{\partial_x}\left(\mathcal L_{V_m} u_m\right)\cdot\sum_{k=K}^\infty \frac{a_k\sin(kx)-b_k\cos(kx)}{k^3}\p(k^2t)
\\&+4{\partial^2_{xx}}u_m\cdot \sum_{k=K}^\infty \frac{a_k\cos(kx)+b_k\sin(kx)}{k^2} \p(k^2t)
\end{align*}
We need one more transformation before this is in a workable form: we observe that (still in the $W^{-1,2}(\O)$ sense) we have
\begin{align*}
{\partial_x}\left(\mathcal L_{V_m} u_m\right)\sum_{k=K}^\infty \frac{a_k\sin(kx)-b_k\cos(kx)}{k^3}\p(k^2t)
=&{\partial_x} \left\{ \mathcal L_{V_m} u_m \sum_{k=K}^\infty \frac{a_k\sin(kx)-b_k\cos(kx)}{k^3} \p(k^2t)\right\}
\\&-\left(\mathcal L_{V_m} u_m\right)\sum_{k=K}^\infty \frac{a_k\cos(kx)+b_k\sin(kx)}{k^2}\p(k^2t),
\end{align*}
and the equation we shall be working on is then given by 

\begin{multline}\displaystyle \mathcal L_{V_m} R_K=-\left(\mathcal L_{V_m} u_m\right)\sum_{k=K}^\infty \frac{a_k\cos(kx)+b_k\sin(kx)}{k^2}(1-e^{-k^2t}+2\p(k^2t))
\\+2{\partial_x} \left\{ \mathcal L_{V_m} u_m \sum_{k=K}^\infty \frac{a_k\sin(kx)-b_k\cos(kx)}{k^3} \p(k^2t)\right\}
\\+4{\partial^2_{xx}}u_m\cdot\sum_{k=K}^\infty \frac{a_k\cos(kx)+b_k\sin(kx)}{k^2} \p(k^2t)
 \end{multline}

\paragraph{Estimating $R_K$}
We introduce the notation 
$$\phi(t):=1-e^{-s}+2\p(s).$$ As $\p$ is bounded, so is $\phi$. Let $R_{K,1}\,, R_{K,2}\,, R_{K,3}$ be the solutions of
\begin{equation}
\begin{cases}
\mathcal L_{V_m} R_{K,1}=\displaystyle-\left(\mathcal L_{V_m} u_m\right)\sum_{k=K}^\infty \frac{a_k\cos(kx)+b_k\sin(kx)}{k^2}\phi(k^2t)\,, 
\\
\\ \mathcal L_{V_m} R_{K,2}=\displaystyle2{\partial_x} \left\{ \mathcal L_{V_m} u_m \sum_{k=K}^\infty \frac{a_k\sin(kx)-b_k\cos(kx)}{k^3} \p(k^2t)\right\}\,, 
\\
\\ \mathcal L_{V_m} R_{K,3}=\displaystyle-4{\partial^2_{xx}}u_m\cdot\sum_{k=K}^\infty \frac{a_k\cos(kx)+b_k\sin(kx)}{k^2} \p(k^2t).\end{cases}\end{equation}

Obviously, 
$$R_K=R_{K,1}+R_{K,2}+R_{K,3},$$ and so, up to a multiplicative constant $E_1$ we have 
\begin{multline}\label{Eq:BP}\iint_\TT|{\partial_x} R_K|^2+\iint_\TT R_K^2+\int_\T R_K^2(T,\cdot)\\\leq E_1\sum_{j=1}^3\left(\iint_\TT |{\partial_x} R_{K,j}|^2+\iint_\TT R_{K,j}^2+\int_\T R_{K,j}^2(T,\cdot)\right).\end{multline}

We shall now estimate each of these three functions. All the upcoming estimates rely on the following, standard, parabolic regularity result (proved in appendix \ref{Ap:Pa}):
\begin{lemma}\label{Le:EstPot}
Let $f\in L^2(\O)$ and $g\in L^2(\O)$. Let $q\in L^\infty(\TT)$. Let $\theta$ be the solution of 
\begin{equation}\begin{cases}
\partial_t\theta-\partial^2_{xx}\theta-V\theta=\partial_x f+qg&\text{ in } \TT\,, 
\\ \theta(0,\cdot)=0.\end{cases}\end{equation} Then
\begin{equation}\iint_\TT |\partial_x\theta|^2+\iint_\TT\theta^2+\int_\T \theta^2(T,\cdot)\leq C(V,q) \iint_\TT (f^2+g^2).\end{equation}
\end{lemma}
We can move back to estimating $R_{K,j}$, for $j=1,2,3$. To estimate $R_{K,1}$ we apply Lemma \ref{Le:EstPot} with 
$$f=0\,, q=-\mathcal L_{V_m} u_m\,, g=\sum_{k=K}^\infty \frac{a_k\cos(kx)+b_k\sin(kx)}{k^2}\phi(k^2t).$$ Thus we obtain 

\begin{multline}\iint_\TT |{\partial_x} R_{K,1}|^2+\iint_\TT R_{K,1}^2+\int_\TT R_{K,1}(T,\cdot)^2\\\leq C_1 \iint_\TT \left\{\sum_{k=K}^\infty \frac{a_k\cos(kx)+b_k\sin(kx)}{k^2}\phi(k^2t)\right\}^2dtdx.\end{multline} Since $\phi$ is bounded by a constant, which we take equal to 1 up to changing the value of $C_1$, we obtain

\begin{equation}\label{Est:RK1}\iint_\TT |{\partial_x} R_{K,1}|^2+\iint_\TT R_{K,1}^2+\int_\TT R_{K,1}(T,\cdot)^2\leq C_1  \sum_{k=K}^\infty \frac{a_k^2+b_k^2}{k^4}.\end{equation}
For $R_{K,2}$, it suffices to apply lemma \ref{Le:EstPot} with 
$$f=2\mathcal L_{V_m} u_m\sum_{k=K}^\infty \frac{a_k\sin(kx)-b_k\cos(kx)}{k^3}\p(k^2t)\,, q=g=0$$ and we obtain, since $\mathcal L_{V_m} u_m\in L^\infty(\TT)$, the existence of a constant $C_2$ such that
\begin{equation}\label{Est:RK2}\iint_\TT |{\partial_x} R_{K,2}|^2+\iint_\TT R_{K,2}^2+\int_\T R_{K,2}(T,\cdot)^2\leq C_2  \sum_{k=K}^\infty \frac{a_k^2}{k^6}.\end{equation}

The case of $R_{K,3}$ is, on the other hand, trickier, but can be handled similarly. We first recall that, from proposition \ref{Pr:Regularity}, for any $p\in [1;+\infty)$, 
\[\mathfrak M(p):=\sup_{t\in[0,T]}\Vert u_m(t,\cdot)\Vert_{W^{2,p}(\T)}<\infty.\]
From standard $W^{1,2}$ parabolic estimates, we obtain, for a constant $C_3$,
\begin{multline}
\iint_\TT |{\partial_x} R_{K,3}|^2+\iint_\TT R_{K,3}^2+\int_\T R_{K,3}^2(T,\cdot)\\\leq C_3\iint_\OT \left({\partial^2_{xx}}u_m\right)^2\left\{\sum_{k=K}^\infty \frac{a_k\cos(kx)+b_k\sin(kx)}{k^2} \p(k^2t)\right\}^2.\end{multline}
Thus, up to replacing $C_3$ with $C_3\Vert \varphi\Vert_{L^\infty}$, we obtain
\begin{multline}
\int_\T R_{K,3}^2(T,\cdot)+\iint_\TT R_{K,3}^2+\iint_\TT |{\partial_x} R_{K,3}|^2\\\leq C_3\iint_\OT \left({\partial^2_{xx}}u_m\right)^2\left\{\sum_{k=K}^\infty \frac{a_k\cos(kx)+b_k\sin(kx)}{k^2} \right\}^2.\end{multline}
Define \[\Psi_K:=\sum_{k=K}^\infty \frac{a_k\cos(kx)+b_k\sin(kx)}{k^2}.\]
From the H\"{o}lder inequality, we obtain
\begin{align*}
\int_\T R_{K,3}^2(T,\cdot)+\iint_\TT R_{K,3}^2+\iint_\TT |{\partial_x} R_{K,3}|^2&\leq C_3\iint_\TT \left({\partial^2_{xx}}u_m\right)^2\Psi_K^2
\\&\leq \underbrace{C_3 \mathfrak M(4)}_{=:C_3'}\int_0^T \Vert {\partial_x} \Psi_K(t,\cdot)\Vert_{L^2(\T)}\Vert \Psi_K(t,\cdot)\Vert_{L^2(\T)}dt
\\&\leq 2{C_3'}\left\{\Upsilon\iint_\TT |{\partial_x} \Psi_K|^2+\frac1\Upsilon\iint_\TT \Psi_K^2\right\}.
\end{align*}
However, since 
\[{\partial_x}\Psi_K=\sum_{k=K}^\infty \frac{-a_k\sin(kx)+b_k\cos(kx)}{k}\] we obtain, on the one-hand, 
\[\iint_\TT \Psi_K^2=\frac12\sum_{k=K}^\infty \frac{a_k^2+b_k^2}{k^4},\] and, on the other hand, 
\[\iint_\TT |{\partial_x} \Psi_K|^2=\frac12\sum_{k=K}^\infty \frac{a_k^2+b_k^2}{k^2}.\]
Finally, we obtain the estimate
\begin{equation}\label{Est:RK3}
\int_\T R_{K,3}^2(T,\cdot)+\iint_\TT R_{K,3}^2+\iint_\TT |{\partial_x} R_{K,3}|^2\leq \Upsilon C_3'\sum_{k=K}^\infty \frac{a_k^2+b_k^2}{k^2}+\frac{C_3'}\Upsilon\sum_{k=K}^\infty \frac{a_k^2+b_k^2}{k^4}.\end{equation}

Hence, summing \eqref{Est:RK1}-\eqref{Est:RK2}-\eqref{Est:RK3} and plugging these estimates in \eqref{Eq:BP}, there exists $C_{\mathrm{cont}}$ such that

\[\iint_\TT |{\partial_x} R_{K}|^2+\iint_\TT R_{K}^2+\int_\T R_{K}^2(T,\cdot)\leq  \frac{C_{\mathrm{cont}}}\Upsilon\sum_{k=K}^\infty\frac{a_k^2+b_k^2}{k^4}+\Upsilon C_{\mathrm{cont}}\sum_{k=K}^\infty\frac{a_k^2+b_k^2}{k^2}
,\]
thus concluding the proof.
\end{proof}
Thus proposition \ref{Pr:Rosecroix} is proved. As, from lemma \ref{Le:Rosecroixsuffices}, proposition \ref{Pr:Rosecroix} implies theorem \ref{Th:Bgbg}, theorem \ref{Th:Bgbg} is established.

\section{Conclusion}
\subsection{Possible generalisations of Theorem \ref{Th:Bgbg}}
\subsubsection{General comment about generalisations}Throughout these generalisations, we still assume that we are working with an initial condition $u^0\in \mathscr C^2(\T)$ with $\inf_\T u^0>0$.
Let us draw attention to the fact that the core idea of the proof of theorem \ref{Th:Bgbg} consists in combining two ingredients: the first one is proposition \ref{Pr:Cruise}, which gives a lower estimate of $\ddot{\mathcal J}$, and the second one is a two scale asymptotic expansion. This second part is independent of the functionals $j_1\,, j_2$, the monotonicity of which are only used in the first step. To generalise our model to other types of interactions, some assumptions will ensure that proposition \ref{Pr:Cruise} remain valid. A crucial part in deriving the conclusion however is estimate \eqref{Est:LK}. To obtain it, we used the fact that in our bilinear model we have $\inf_{\TT}|u_m|>0$. This plays a role when using the fact that $\dot u_m$ solves 
\[{\partial_t}\dot u_m-{\partial^2_{xx}} \dot u_m-V_m\dot u_m=u_m h.\] In other types of model, $\dot u_m$ solves (generically) an equation of the form
\[{\partial_t} \dot u_m-{\partial^2_{xx}} \dot u_m-V_m\dot u_m=F(u_m,m) h,\] and other assumptions will thus ensure that $\inf_\TT F(u_m,m)>0$. 

\subsubsection{Approximations of time dependent controls}\label{Se:TDM}
Although we can not handle general time-dependencies, see section \ref{Se:TimeDependent}, we would nonetheless like to draw attention to the fact that our method covers some approximations of time-varying controls.  Consider an integer $N$ and a family of functions $\{\phi_i\}_{i=1,\dots,N}$ satisfying the following conditions:
\begin{equation}\label{A1}\tag{$\bold{A}_1$}\text{ For any $i\in \{1,\dots,N\}$, }\phi_i\in\mathscr C^1([0,T],\R)\end{equation}
\begin{equation}\label{A2}\tag{$\bold{A}_2$}\text{For any $i\in \{0,\dots,N\}$, }\inf_{[0,T]}\left|\phi_i\right|>0.\end{equation}

 Let 
\begin{multline}\mathcal M_N(\T):=\Big\{m\in L^\infty(\TT) \text{ that write }m=\sum_{i=1}^{N}\phi_{i}(t)m_{i} \\\text{ where for any $i\in \{1,\dots,N\}\,, m_i\in \mathcal M(\T)$}\Big\}.\end{multline} 
A generic  $\overline m\in \mathcal M_N(\T)$ is identified  with the associated $N$-tuple $(m_1,\dots,m_N)\in \mathcal M(\T)^N$.  A function $\overline m\in\mathcal M_N(T)$ is called \emph{bang-bang} if for any $i\in \{1,\dots,N\}$ $m_i$ is a bang-bang function.

We can define $u_{\overline m}$ as the solution of \eqref{Eq:MainBil} with $m$ replaced with $\overline m$, and the optimisation problem is

\begin{equation}\label{Eq:PvN}\tag{$\bold P_{\mathrm{parab}}^N$}\fbox{$\displaystyle
\max_{\overline m\in \mathcal M_N(\T)}\mathcal J(\overline m)$,}\end{equation} where $\mathcal J$ still satisfies  assumption \eqref{Eq:HJ} of Theorem \ref{Th:Bgbg}. We claim that, up to minor adaptations of our proof, the following result holds:
\begin{theorem}\label{Th:BgbgN}
Assume $\mathcal J$ satisfies \eqref{Eq:HJ} and $\phi=\{\phi_i\}_{i=1\,, \dots,N}$ satisfies \eqref{A1}-\eqref{A2}. Any solution $\overline m^*$ of \eqref{Eq:PvN} is bang-bang: there exist $E_1\,, \dots,E_N\subset \T$ such that 
\[\overline m^*=\sum_{i=0}^N \phi_i(t)\mathds 1_{E_i}.\]
\end{theorem}
\paragraph{Sketch of proof of theorem \ref{Th:BgbgN}}
First of all, we once again have 
\[\inf_\TT u_\om>0\text{ and for any $p\in[1;+\infty)$ } \sup_{t\in[0,T]}\Vert u_\om(t,\cdot)\Vert_{W^{2,p}(\T)}<\infty.\]

We can compute, for an admissible $\overline m\in \mathcal M_N(\T)$ and an admissible perturbation $\overline h=(h_1,\dots,h_N)$ at $\overline m$, the first and  second order derivatives of $\overline m \mapsto u_{\overline m}$ in the direction $\overline h$ solve, respectively,
\begin{equation}\label{Eq:DotuTD}
\begin{cases}
{\partial_t} \dot u_{\overline m}-{\partial^2_{xx}} \dot u_{\overline m}-V_{\overline m}\dot u_{\overline m}= u_{\overline m}\sum_{i=1}^N \phi_i(t)h_i&\text{ in }\TT,
\\ \dot u_{\overline m}(0,\cdot)\equiv 0 &\text{ in }\T.
\end{cases}
\end{equation}
and 
\begin{equation}\label{Eq:DdotuTD}
 \begin{cases}
 {\partial_t} \ddot u_{\overline m}-{\partial^2_{xx}} \ddot u_{\overline m}-V_{\overline m}\ddot u_{\overline m}=2 \dot u_{\overline m}\sum_{i=1}^N \phi_i(t)h_i+\left.{\partial^2_{uu}}f\right|_{u=u_{\overline m}} \left(\dot u_{\overline m}\right)^2&\text{ in }\TT,
\\ \ddot u_{\overline m}(0\cdot)\equiv 0&\text{ in }\T
 \end{cases}
 \end{equation}
with 
\[V_{\overline{m}}:=\left({\overline m}+\left.{\partial_u}f\right|_{u=u_{\overline m}}\right).\]
We introduce the adjoint state $p_{\overline{m}}$, solution of
\begin{equation}\label{Eq:AdjointTD}\begin{cases}
{\partial_ t} p_\om+{\partial^2_{xx}} p_\om+V_\om p_\om=-\left.{\partial_u}j_1\right|_{u=u_\om}\text{ in }\TT\,, 
\\ p_\om(T,\cdot)=\left.{\partial_u}j_2\right|_{u=u_\om} &\text{ in }\T.
\end{cases}\end{equation}
From the same arguments as in lemma \ref{Le:RegAdjoint} we have 
\[\forall \e>0\,, \inf_{\TeT}p_\om>0\text{ and for any $p\in[1;+\infty)$ } \sup_{t\in[0,T]}\Vert p_\om(t,\cdot)\Vert_{W^{2,p}(\T)}<\infty.\] For any admissible perturbation $\overline h$ we then have 

 \begin{multline}\ddot{\mathcal J}(\om)[\overline h,\overline h]=2\iint_\TT \dot u_\om p_\om\left\{\sum_{i=1}^N \phi_i(t)h_i(t)\right\}+\int_\T \dot u_\om^2(T,\cdot) \left.{\partial^2_{uu}}j_2\right|_{u=u_\om(T,\cdot)}
  \\+\iint_\TT \dot u_\om^2\left( \left.{\partial^2_{uu}}j_1\right|_{u=u_\om} +\left.{\partial^2_{uu}}f\right|_{u=u_\om}p_\om\right).
 \end{multline} 
 We then use the fact that 
 \[\sum_{i=1}^N \phi_i(t)h_i
=\frac{{\partial_t} \dot u_{\overline m}-{\partial^2_{xx}} \dot u_{\overline m}-V_{\overline m}\dot u_{\overline m}}{ u_{\overline m} }.\]
From this point on, we can follow all the steps of the proof of proposition \ref{Pr:Cruise} to obtain the existence of three positive constants $\alpha\,, \beta\,, \gamma$ and of a positive $\e>0$ such that, for any admissible perturbation $\overline h$ there holds
\[\ddot{\mathcal J}(\om)[\overline h,\overline h]\geq  \alpha\iint_\TeT |{\partial_x} \dot u_\om|^2-\beta\iint_\TT \dot u_\om^2-\gamma\int_\T \dot u_\om^2(T,\cdot).\]

We  then argue by contradiction, assuming that there exists a maximiser $\om$ and an index $j\in \{1,\dots,N\}$ such that $m_j^*$ is not bang-bang, so that $\omega:=\{0<m_j^*<1\}$ has positive measure. We fix this index $j$ and henceforth only consider perturbations $\overline{h}$ of the form $(0,\dots,h_j,\dots,0)$ with $h_j$ an admissible perturbation at $m_j$ supported in $\omega_j$, and admitting the Fourier decomposition 
\[h_j=\sum_{k=K}^\infty a_k\cos(kx)+b_k\sin(kx).\] For such a perturbation, the derivative $\dot u_\om$ solves 
\begin{equation}\label{Eq:DotuTDD}
\begin{cases}
{\partial_t} \dot u_{\overline m}-{\partial^2_{xx}} \dot u_{\overline m}-V_{\overline m}\dot u_{\overline m}= u_{\overline m}\phi_j(t)h_j\text{ in }\TT,
\\ \dot u_{\overline m}(0,\cdot)\equiv 0 &\text{ in }\T,
\end{cases}
\end{equation}
and given that $\inf_{\TT}|u_\om\phi_j|>0$ we can conclude in exactly the same way.

\subsubsection{Other types of interactions: generalisation and obstruction}\label{Se:Interaction}
One may argue that other types of interactions can be relevant. To motivate this point, let us consider another type of model from spatial ecology, were one rather aims at optimising a certain criterion for a state equation of the form 
\begin{equation}\begin{cases}
{\partial_t} y_m-{\partial^2_{xx}} y_m=f(t,x,y_m)+m\p(y_m)&\text{ in }\TT\,, 
\\ y_m(0,\cdot)=y^0&\text{ in }\T,\end{cases}\end{equation} where $y^0\in \mathscr C^2(\T)$ and $\inf_\T y^0>0$ is a fixed initial condition, $f$ and $\p$ are non-linearities that must satisfy that for any $m\in \mathcal M(\T)$, $y_m$ satisfies
\begin{equation}\inf_{\TT} y_m>0\end{equation} and 
\begin{equation}\forall T>0\,, \sup_{t\in [0,T]}\Vert y_m(t,\cdot)\Vert_{L^\infty}<\infty.\end{equation}
We aim at optimising 
\begin{equation}\mathcal J:\mathcal M(\T)\ni m\mapsto \iint_\TT j_1(t,x,y_m)+\int_\T j_2(x,y_m(T,\cdot))\end{equation} and assume that $\mathcal J$ satisfies \eqref{Eq:HJ}.

We claim that, up to minor adaptation of the proof of theorem \ref{Th:Bgbg}, the following result holds:
\begin{theorem}\label{Th:BgbgP}
Assume that $\p$ is $\mathscr C^2$ on $\R_+^*$. If, for any $K\in \R_+^*$, for any $\e>0$,
\begin{equation}\inf_{y\in(0,K)}\frac{\p'(y)}{\p(y)}=a_1(K)>0\,, \inf_{y\in (\e;K)}|\p(y)|=a_2(\e,K)>0\end{equation} then any solution $m_\p^*$ of the optimisation problem
\[ \sup_{m\in \mathcal M(\T)}\mathcal J(m)\] is bang-bang.
\end{theorem}
Let us explain why this type of setting is relevant in application: consider the case of the logistic-diffusive equation 
\begin{equation}\begin{cases}
{\partial_t} y_m-{\partial^2_{xx}} y_m=y_m(1-my_m)&\text{ in }\TT\,, 
\\ y_m(0,\cdot)=y^0&\text{ in }\T,\end{cases}\end{equation} as well as the functional 
\[ \mathcal J(m):=\int_\T y_m(T,\cdot).\] Maximising $\mathcal J$ with respect to $m\in \mathcal M(\T)$ amounts to optimising the total population size with respect to $m$, the inverse of the carrying capacity (at this stage, one may argue that it would make more sense to consider the case of $m$ satisfying $\e\leq m\leq 1$ in the definition of $\mathcal M(\T)$; given remark \ref{Re:Constraints}, we claim that this would not change anything to the conclusion of theorem \ref{Th:BgbgP}). Such a problem is inspired by the considerations of \cite{DeAngelis2015}.
\paragraph{Sketch of proof of theorem \ref{Th:BgbgP}}
We start by noticing that
\[\inf_\TT y_m>0\text{ and for any $p\in[1;+\infty)$ } \sup_{t\in[0,T]}\Vert y_m(t,\cdot)\Vert_{W^{2,p}(\T)}<\infty.\]

Let  $m\in \mathcal M(\T)$, and consider an admissible perturbation $h$ at $ m$; the first and  second order derivatives of $m \mapsto u_{m}$ in the direction $ h$ solve,
\begin{equation}\label{Eq:DotuTD}
\begin{cases}
{\partial_t} \dot y_m-{\partial^2_{xx}} \dot y_{m}-V_{m}\dot y_{m}= \p(y_{m})h&\text{ in }\TT,
\\ \dot y_{m}(0,\cdot)\equiv 0 &\text{ in }\T.
\end{cases}
\end{equation}
and 
\begin{equation}\label{Eq:DdotuTD}
 \begin{cases}
 {\partial_t} \ddot y_m-{\partial^2_{xx}} \ddot y_{m}-V_{m}\ddot y_{m}=2 \dot y_{m}h\p'(y_m)+\left(m\p''(y_m)+\left.{\partial^2_{uu}}f\right|_{u=y_{m}} \right)\dot y_{m}^2&\text{ in }\TT,
\\ \ddot y_{m}(0\cdot)\equiv 0&\text{ in }\T
 \end{cases}
 \end{equation}
with 
\[V_{m}:=\left({m}+\left.{\partial_{u}}f\right|_{u=y_{m}}+\p'(y_m)\right).\]
We introduce the adjoint state $q_{m}$, solution of
\begin{equation}\label{Eq:AdjointTD}\begin{cases}
{\partial_t}q_m+{\partial^2_{xx}} q_m+V_m q_m=-\left.{\partial_u}j_1\right|_{u=y_m}\text{ in }\TT\,, 
\\ q_m(T,\cdot)=\left.{\partial_u}j_2\right|_{u=y_m} &\text{ in }\T.
\end{cases}\end{equation}
From the same arguments as in lemma \ref{Le:RegAdjoint} we have 
\[\forall \e>0\,, \inf_{\TeT}q_m>0\text{ and for any $p\in[1;+\infty)$ } \sup_{t\in[0,T]}\Vert q_m(t,\cdot)\Vert_{W^{2,p}(\T)}<\infty.\] For any admissible perturbation $\overline h$ we then have 

 \begin{multline}\ddot{\mathcal J}(m)[\overline h,\overline h]=2\iint_\TT \dot y_m \p'(y_m)q_mh+\int_\T \dot y_m^2(T,\cdot) \left.{\partial^2_{uu}}j_2\right|_{u=y_m(T,\cdot)}
  \\+\iint_\TT \dot y_m^2\left( \left.{\partial^2_{uu}}j_1\right|_{u=y_m} +\left.{\partial^2_{uu}}f\right|_{u=y_m}q_m+m\p''(y_m)\right).
 \end{multline} 
 We then use the fact that 
 \[h
=\frac{{\partial_t} \dot y_m-{\partial^2_{xx}} \dot y_{m}-V_{m}\dot y_{m}}{ \p(y_{m}) }.\]
Following all the steps of the proof of proposition \ref{Pr:Cruise}, we obtain the existence of two constants $\beta\,, \gamma$ such that, for any admissible perturbation $h$ at $m$, there holds
\[\ddot{\mathcal J}(m)[ h, h]\geq  \iint_\TT\frac{q_m\p'(y_m)}{\p(y_m)} |{\partial_x} \dot y_m|^2-\beta\iint_\TT \dot y_m^2-\gamma\int_\T \dot y_m^2(T,\cdot).\]
We then use the assumption to obtain the existence of an $\e>0$ such that we have, for any admissible perturbation $h$ at $m$, the estimate
\begin{multline*}\ddot{\mathcal J}(m)[ h, h]\\\geq \left(\inf_{\TeT}q_m\right)\frac{a_1(\sup_\TT y_m)}{a_2(\inf_\TT y_m,\sup_\TT y_m)} \iint_\TeT\ |{\partial_x} \dot y_m|^2-\beta\iint_\TT \dot y_m^2-\gamma\int_\T \dot y_m^2(T,\cdot)
\\\geq \alpha \iint_\TeT\ |{\partial_x} \dot y_m|^2-\beta\iint_\TT \dot y_m^2-\gamma\int_\T \dot y_m^2(T,\cdot) \end{multline*} for a positive $\alpha>0$.
We  then follow exactly the same steps.

\subsubsection{Some interactions not covered by our method}
However, despite their interest, our generalisations, theorems \ref{Th:BgbgP} and \ref{Th:BgbgN}, do not cover several cases. A typical example of such an interaction between the state and the control is, typically, of the form $m\p(u_m)$ with $\p$ an increasing, negative function. Indeed, the following is easily checked \emph{via} the same computations: let $\p$ be a smooth function such that
\[\forall K\,, \e>0\,, \sup_{(0;K)}\frac{\p'}{\p}=-a_1<0\,, \inf_{[\e;K]} |\p|=a_2(\e,K)>0,\] and consider the solution $z_m$ of
\begin{equation}\begin{cases}
{\partial_t}z_m-{\partial^2_{xx}} z_m=f(t,x,z_m)+m\p(z_m)&\text{ in }\TT\,, 
\\ z_m(0,\cdot)=z^0&\text{ in }\T,\end{cases}\end{equation} 
where $\inf_\T z^0>0$ and $\p$ and $f$ are further chosen to satisfy 
\[ \forall m\in \mathcal M(\T)\,, \inf_\TT z_m>0\,, \sup_{t\in[0,T]}\Vert z_m(t,\cdot)\Vert_{W^{2,p}(\T)}<\infty.\] Then, considering a functional $\mathcal J$ that satisfies \eqref{Eq:HJ}, there exists a positive constant $\alpha>0$ and two constants $\beta\,, \gamma>0$ such that
\[\ddot{\mathcal J}(m)[h,h]\leq- \alpha\iint_\TeT |{\partial_x}\dot z_m|^2+\beta\iint_\TT \dot z_m^2+\gamma\int_\T \dot z_m^2(T,\cdot),
\]where $\dot z_m$ is of course the derivative of $m\mapsto z_m$ at $m$ in the direction $h$. In this case, our method fails to provide a bang-bang property for maximisers, but yields a bang-bang property for minimisers.

\subsection{Generalisations and obstructions for Theorem \ref{Th:Bgbg}}
In this section, we present several possible obstructions and generalisations of theorem \ref{Th:Bgbg} and of our methods to other contexts (\emph{e.g} to the multi-dimensional case or to other geometries), pinpointing what the main difficulties seem to be.
\subsubsection{Higher dimensional tori}\label{ff}
We believe that our method extends, in a straightforward manner, to the case of $d$-dimensional tori, for any $d\in \N\backslash\{0\}$; once again, two steps are crucial in deriving theorem \ref{Th:Bgbg}. The first one, proposition \ref{Pr:Cruise}, is an estimate on the second order Gateaux derivative of the functional which does not depend on the dimension, see in particular remark \ref{Argued}. The second one is to establish a two-scale asymptotic expansions for solutions of a linear heat equation with a highly oscillating source term. We claim that this step can be extended in a straightforward way to the $d$ dimensional torus, provided the functions $\cos(k\cdot)$ and $\sin(k\cdot)$ are replaced with products of the form $\prod_{i=1}^d \cos(k_ix)\sin(k_i'x)$ or $\prod_{i=1}^d\cos(k_ix)$ where $(k_1,\dots\,,k_d\,,k_1',\dots\,,k_d')\in \N^{2d}$. 

\subsubsection{Possible obstructions in other domains}
It would be extremely interesting and relevant, in many applications, to consider not only the case of bounded domains $\O\subset \R^d$ with, for instance, Neumann or Robin boundary conditions (Dirichlet boundary conditions may not be suitable for our needs, as we need, in a crucial manner, a uniform lower bound on solutions of the equation). In this context, we claim that the lower estimate given by proposition \ref{Pr:Cruise} still holds, as is clear in the proof and in remark \ref{Argued}. The main difficulty lies elsewhere, namely, in the possibility to attain two-scale asymptotic expansions in order to derive the bang-bang property. A possibility to do so would be to replace the $\cos(k\cdot)\,, \sin(k\cdot)$ with $\psi_k$, where the $\{\psi_k\}_{k\in \N}$ or the Neumann (if working with Neumann boundary conditions) or Robin (if working with Robin boundary conditions) eigenfunctions of the laplacian in $\O$, associated with the (increasing) sequence of eigenvalues $\{\lambda_k\}_{k\in \N}$. Let us assume that we are working with Neumann boundary conditions. Then the task at hand would be, if we mimicked our approach, to find an asymptotic expansion for a solution $\dot u_m$ of 
\[
\begin{cases}
{\partial_t}\dot u_m-\Delta \dot u_m-V_m \dot u_m=u_m(t,x)\sum_{k=K}^\infty a_k \psi_k(x)&\text{ in }\OT\,, 
\\{\partial_\nu}\dot u_m=0&\text{ on }\partial \O\,, 
\\ \dot u_m(0,\cdot)=0&\text{ in }\O.\end{cases}\] It is unclear, in this situation, which asymptotic expansion would yield a result analogous to that of proposition \ref{Pr:Rosecroix}.

\subsubsection{Possible obstructions for other diffusion operators}
A very relevant query, if we keep application to mathematical biology in mind, is the analysis of heterogeneous diffusion operators. In other words, following, for instance, \cite{BelgacemCosner}, one may rather be interested in state equations assuming the form 
\[ {\partial_t}u_m-\nabla \cdot\left(A\nabla u_m\right)=mu_m+f(t,x,u)\] where $A=A(t,x)$ accounts for some heterogeneity.  It is likely that our methods extend to this case, provided $A$ is smooth enough to guarantee uniform (in time) $W^{2,p}(\T)$ (in space) estimates on the solution $u_m$. Although this can be of interest, since our main goal is the analysis of shape optimisation problems in optimal control problems, this is not the most relevant analytical setup for us. Indeed, in the context of spatial ecology, the diffusion matrix $A$ and the resources distribution $m$ are often linked, which will lead to very intricate situations from the regularity point of view. For an example of such an optimisation problem and of the wealth of qualitative and technical issues it can lead to we refer to the elliptic optimisation problem studied in \cite{MNPARMA}.

\subsubsection{Some possible open questions}
Of course, the bang-bang property is one of the many qualitative aspects of bilinear optimal control problems. Even when $m$ does not depend on time, and we know that maximisers are bang-bang, what do {these} optimisers look like from a geometric point of view? In other words, considering a maximiser $m^*=\mathds 1_{E^*}$, what are the geometric and topological features of $E^*$? Is it connected, disconnected, and how may we quantify such information? Let us underline here that the functionals under consideration in this paper are non-energetic, which prohibits the use of rearrangement techniques. Such techniques, developed in the context of mathematical biology in \cite{BHR} for instance, although very powerful for energetic or spectral optimisation problems in the elliptic case \cite{LamboleyLaurainNadinPrivat} or for the study of concentration phenomena in parabolic models \cite{Alvino1986,alvino1991,Alvino1990,Bandle,MNLA,Mossino,RakotosonMossino,Vazquez} can not yield satisfying results for non-energetic problems. As an example, we refer to the elliptic problem of optimising the total population size described in section \ref{Se:Bib} and, more specifically, to the results of \cite{Heo2021,MNPCPDE,MRBSIAP}: in the elliptic context, depending on the dispersal rate of the population, optimal resources distributions are either concentrated (and when that dispersal rate is high enough we can apply symmetrisation properties) or display hectic oscillations (this corresponds to the low dispersal rate limit, and is clearly a case where it is hopeless to apply rearrangements). The study of these properties in parabolic models seems challenging, and we plan on studying it in further works.

\subsection{The difficulty with general time dependent controls}\label{Se:TimeDependent}
Finally, let us underline the core difficulties in reaching the bang-bang property for general time-dependent controls. In other words, assume we are working with controls $m$ satisfying
\begin{multline*}
m\in \mathcal M(\TT)\\=\left\{m\in L^\infty(\TT):\, 0\leq m\leq 1\text{ a.e.  and for a.e. $t\in [0,T]$ }\int_\T m(t,\cdot)=m_0\right\},\end{multline*}
and we are studying the maximisation of $\mathcal J$ over $\mathcal M(\TT)$. Two problems rapidly arise. The first one is that, as noted several times, the regularity of $u_m$ is crucial in deriving proposition \ref{Pr:Cruise}; this necessarily requires some \emph{ a priori } assumptions on the regularity of the control $m$ in time. The second difficulty is in defining, for a given maximiser $m^*$ that would, arguing by contradiction, not be bang-bang, a highly oscillating perturbation $h$. One may be tempted to choose a perturbation $h$ supported in the right set (\emph{i.e.} in the set $\{0<m^*<1\}$) by reasoning as follows: define, for a.e. $t\in [0,T]$, $\omega(t):=\{t\}\times \{0<m^*<1\}$ and consider a function $h_t$ supported in $\omega(t)$ that only has high (enough) Fourier modes. One would then define the perturbation as $h(t,x)=h_t(x)$. The problem here is that there is no guarantee that such a function $h$ is measurable in $(t,x)$.

On the other hand, enforcing strong time regularity constraints on the control $m$ (such as, for instance, $m\in \mathscr C^1([0,T];\mathcal M(\T))$) may allow such constructions to work. This is however, beyond the scope of our article, and we plan on investigating the influence of time-regularity constraints on the bang-bang property {future works}.

\bibliographystyle{abbrv}

\bibliography{BiblioMNPParab}

\appendix
\section{Study of the parabolic model}\label{Ap:Parab}
\subsection{Proof of lemma \ref{Le:ubdd}}\label{Ap:ubdd}
\begin{proof}[Proof of lemma \ref{Le:ubdd}]
Given that $f$ satisfies \eqref{H3}, $u_m$ satisfies the inequality
\begin{equation}\begin{cases}
{\partial_t}u_m-{\partial^2_{xx}} u_m\geq mu_m+f(t,x,0)-A|u_m|\geq mu_m-A|u_m|&\text{ in }\TT
\\u_m(0,\cdot)=u^0&\text{ in }\T.\end{cases}\end{equation}
Multiplying this equation by the negative part $u_m^-$ of $u_m^-$ and integrating by parts gives
\[\frac{d}{d t}\int_\T (u_m^-)^2+\int_\T (u_m^-)^2\leq A\int_\O (u_m^-)^2\] and since $u_m^-(0,\cdot)\equiv 0$ we obtain 
\[u_m\geq 0\text{ in }\TT.\] To derive that 
\[\inf_\TT u_m>0\] it suffices to apply the strong maximum principle.

For the upper bound, let $\kappa$ be given by \eqref{H2}. Up to replacing $\kappa$ with $\max\{\Vert u^0\Vert_{L^\infty},\kappa\}$ we may assume that 
\[\kappa \geq \Vert u^0\Vert_{L^\infty}.\] Let $z_m:=u_m-\kappa$. $z_m$ solves the partial differential equation
\begin{align*}{\partial_t}z_m-{\partial^2_{xx}} z_m=mu_m+f(t,x,u_m)&=mz_m+f(t,x,u_m)+\kappa m
\\&=mz_m+f(t,x,u_m)-f(t,x,\kappa)+\kappa m+f(t,x,\kappa)
\\&\leq mz_m+f(t,x,u_m)-f(t,x,\kappa)
\\&\leq (m+A)|z_m|.
\end{align*}
As $z_m(0,\cdot)\leq 0$, the conclusion follows: $z_m\leq 0$ in $\TT$ so that $u_m\leq \kappa$.
\end{proof}

\subsection{Regularity results: proof of proposition \ref{Pr:Regularity}}\label{Ap:Regularity}
To prove proposition \ref{Pr:Regularity} we need an auxilliary result:

\begin{lemma}\label{Le:Lp}For any $p\in [1;+\infty)$ there exists a constant $C_p$ such that the following holds: there exists $q>1$ such that, for any $\theta^0\in \mathscr C^\infty(\T)$, for any ${F}\,, W\in W^{1,q}(0,T;L^q(\T))\cap L^\infty(\TT)$, if  $\theta$ be the unique $L^p(0,T;W^{1,p}(\T))$ solution of
\begin{equation}\label{Eq:theta}\begin{cases}
\partial_t \theta-\partial^2_{xx} \theta-W \theta={F}&\text{ in }\TT\,, 
\\ \theta(0,\cdot)=\theta^0&\text{ in }\T.\end{cases}\end{equation} Then 
\begin{multline}\label{Eq:EstLp}
\sup_{t\in [0,T]}\left\Vert \partial_t \theta(t,\cdot)\right\Vert_{L^p(\T)}+\sup_{t\in[0,T]}\left\Vert \theta(t,\cdot)\right\Vert_{W^{2,p}(\T)}\\\leq C_p\left(\left\Vert{F}\right\Vert_{W^{1,q}(0,T;L^q(\T))}+\Vert W\Vert_{W^{1,q}(0,T;L^q(\O)}+\left\Vert \theta^0\right\Vert_{\mathscr C^2(\T)}\right). \end{multline}It suffices to take $q=4\lfloor p\rfloor.$

\end{lemma}
\begin{proof}[Proof of Lemma \label{Le:Lp}]
Let us first prove that for any $p\in [1;+\infty)$ there exists a constant $C_p^0$ such that
\begin{equation}\label{Eq:Sayre}
\sup_{t\in [0,T]}\Vert \theta(t,\cdot)\Vert_{L^p(\T)}\leq C_p^0\left(\Vert {F}\Vert_{L^\infty(\TT)}+\Vert \theta^0\Vert_{L^q}\right),\end{equation} for $q$ large enough. We first recall the inclusion of Lebesgue spaces: if $p_1>p_0$ then $L^{p_1}(\T)\hookrightarrow L^{p_0}(\T)$. By this inclusion of Lebesgue spaces it suffices to prove \eqref{Eq:Sayre} for $p=2k$, $k=1,\dots,n,\dots$. Let $k\in \N\backslash\{0\}.$
To obtain \eqref{Eq:Sayre} for $p=2k$ we use 
$$v:=2k\theta^{2k-1}$$ as a test function in the weak formulation of \eqref{Eq:theta}. We obtain, for a.e. $t\in(0,T)$, 
\[
\int_\T 2k \left(\partial_t\theta\right)\theta^{2k-1}+2k(2k-1)\int_\T \theta^{2k-2} \left| {\partial_x} \theta\right|^2-2k\int_\T W\theta^{2k}=2k\int_\T {F}\theta^{2k-1}.\]
In particular we obtain
\[ \partial_t\int_\T \theta^{2k}-\Vert W\Vert_{L^\infty(\TT)}\int_\T \theta^{2k}\leq 2k\int_\T \left|{F}\right| \cdot |\theta|^{2k-1}.
\]
We bound the right-hand side using H\"{o}lder's inequality:
\begin{align*}
 2k\int_\T \left|{F}\right| \cdot |\theta|^{2k-1}&\leq 2k\Vert {F}(t,\cdot)\Vert_{L^{2k}(\T)}\left(\int_\T \theta^{2k}\right)^{\frac{2k-1}{2k}}
\end{align*}
Defining $c_0:=\Vert W\Vert_{L^\infty(\TT)}\,, c_1(t):=2k\Vert {F}(t,\cdot)\Vert_{L^{2k}(\T)}$ and setting 
$$y(t):=\int_\T \theta^{2k}(t,\cdot)$$ we are left with the differential inequality
$$y'(t)-c_0 y(t)\leq c_1(t)y(t)^{1-\frac1{2k}},$$ which in turn yields
$$y' y^{\frac1{2k}-1}-c_0 y^{\frac1{2k}}\leq c_1.$$
We set $\overline z=y^{\frac1{2k}}e^{-c_0t}$, which leads to
$$\overline z'(t)\leq c_1(t)e^{-c_0t},$$ and it suffices to integrate this inequality to obtain

$$y_k(t)^{\frac1{2k}}\leq e^{c_0t}y_k(0)^\frac{1}{2k}+2ke^{c_0t}\int_0^t e^{-c_0\tau} \Vert  {F}(\tau,\cdot)\Vert_{L^{2k}(\T)}d\tau.
$$
We bound brutally $t\leq T$ in the exponentials, and we obtain, for a constant $c_3$,
$$
y_k(t)^{\frac1{2k}}\leq  2kc_3\left(y_k(0)^\frac{1}{2k}+\int_0^t  \Vert  {F}(\tau,\cdot)\Vert_{L^{2k}(\T)}d\tau\right).$$
Finally, we use Jensen's inequality:

$$\int_0^t  \Vert  {F}(\tau,\cdot)\Vert_{L^{2k}(\T)}d\tau\leq t^{1-\frac2k}\left(\int_0^t \int_\T {F}^{2k}\right)^{\frac1{2k}}\leq T^{1-\frac2k}\Vert {F}\Vert_{L^{2k}(\TT)}. $$Finally:

$$\Vert \theta(t,\cdot)\Vert_{L^p(T)}\leq 2kc_3\left(\Vert \theta^0\Vert_{L^{2k}(\T)}+T^{1-\frac2k}\Vert {F}\Vert_{L^{2k}(\TT)}\right).$$

To derive \eqref{Eq:EstLp} we differentiate \eqref{Eq:theta} with respect to time. It appears that $q:=\partial_t \theta$ solves
\begin{equation}\label{Eq:q}
\begin{cases}
\partial_t q-{\partial^2_{xx}} q-Wq=\partial_t {F}+\theta \partial_t W&\text{ in }\TT\,, 
\\ q(0,\cdot)=W \theta^0+{\partial^2_{xx}} \theta^0+{F}(0,\cdot)&\text{ in }\T.\end{cases}\end{equation}
We reason once again using only $p=2k\,, k\in \N\backslash\{0\}$.
 It suffices to apply \eqref{Eq:Sayre} to obtain 
\begin{multline}\sup_{t\in (0,T)} \Vert q(t,\cdot)\Vert_{L^{2k}(\T)}\\\leq c_3\left( \Vert W \theta^0\Vert_{L^{2k}(\T)}+\Vert {\partial^2_{xx}} \theta^0\Vert_{L^{2k}(\T)}+\Vert {F}(0,\cdot)\Vert_{L^{2k}(\T)}+T^{1-\frac2k}\Vert \partial_t{F}+\theta\partial_t W\Vert_{L^{2k}(\TT)}\right)
\end{multline}
We bound the first terms as follows:
\[\Vert W \theta^0\Vert_{L^{2k}(\T)}+\Vert {\partial^2_{xx}} \theta^0\Vert_{L^{2k}(\T)}+\Vert {F}(0,\cdot)\Vert_{L^{2k}(\T)}\leq \Vert W \theta^0\Vert_{L^{\infty}(\TT)}+\Vert  \theta^0\Vert_{\mathscr C^2(\T)}+\Vert {F}(0,\cdot)\Vert_{L^{\infty}(\TT)}.
\]
It remains to bound
\[\Vert \partial_t{F}+\theta\partial_t W\Vert_{{L^{2k}(\TT)}}.\]
To control this term we use the arithmetic-geometric inequality to obtain
\begin{align*}\Vert \theta\partial_t W\Vert_{L^{2k}(\TT)}&\leq \frac12\Vert  \theta^2+\left(\partial_t W\right)^2\Vert_{L^{2k}(\TT)}\\&\leq \frac12\left(\Vert \theta\Vert_{L^{4k}(\TT)}^2+\Vert \partial_t W\Vert_{L^{4k}(\TT)}^2\right) \end{align*}
Thus, 
\[\sup_{t\in (0,T)} \Vert \partial_t\theta(t,\cdot)\Vert_{L^p(\T)}<C\] and the constant $C$ only depends on the $W^{1,p}(0,T;L^r(\T))$ norms of all the functions involved.
To obtain the uniform $W^{2,p}$ estimate, we simply observe that for a.e. $t\in (0,T)$ $\theta(t,\cdot)$ solves 
\[-{\partial^2_{xx}} \theta(t,\cdot)=G:={F}-\partial_t\theta+W\theta\] and to apply standard $W^{2,p}$ elliptic regularity estimates.

\end{proof}

We can now give the proof of Proposition \ref{Pr:Regularity}
\begin{proof}[Proof of Proposition \ref{Pr:Regularity}]
We observe that $u_m$ solves
$$\partial_t u_m-{\partial^2_{xx}}u_m-m u_m=f(t,x,u_m).$$ Now define
\[q:=\partial_t u_m.\] As $m$ does not depend on time, we obtain the following equation on $q$:
\[
\begin{cases}
{\partial_t} q-\partial^2_{xx} q-\left(m+{\partial_uf}(t,x,u_m)\right)q={\partial_t}f(t,x,u_m)&\text{ in }\TT
\\ q(0,\cdot)={\partial^2_{xx}} u^0+mu^0+f(0,x,u^0).
\end{cases}\]
As $u_m\leq K$, $\partial_uf(t,x,u_m)\,, \partial_tf(t,x,u_m)\in L^\infty(\TT)$. From Lemma \ref{Le:Lp} we obtain, for any $p\in [1;+\infty)$, 
\[
\sup_{t\in [0,T]}\Vert \partial_t u_m\Vert_{L^p(\T)}+\sup_{t\in[0,T]}\Vert u_m\Vert_{W^{2,p}(\T)}<\infty.
\]

In the same way, we derive the desired estimate on $p_m$.
\end{proof}

\subsection{Proof of lemma \ref{Le:EstPot}}\label{Ap:Pa}
\begin{proof}[Proof of lemma \ref{Le:EstPot}]
Multiplying the equation by $\theta$, integrating by parts in space and using the fact that $V\in L^\infty(\TT)$, there exists a constant $M>0$ such that
\begin{align*}
\frac12\frac{d}{dt} \int_\T \theta^2+\int_\T|{\partial_x} \theta|^2-\frac{M}{2}\int_\T \theta^2&\leq \frac12\frac{d}{dt} \int_\T \theta^2+\int_\T|{\partial_x} \theta|^2-\int_\T V\theta^2
\\&=\int_\T \theta \partial_x f+\int_\T  \theta q g
\\&=-\int_\T f\partial_x \theta+\int_\T \theta q g
\\&\leq \int_\T \frac{f^2+(\partial_x\theta)^2}2+\frac{\Vert q\Vert_{L^\infty}}2\int_\T g^2+\frac{\Vert q\Vert_{L^\infty}}2\int_\T \theta^2.
\end{align*}
Thus, there exists two constants $M'\,, M''$ such that
$$\frac{d}{dt} \int_\T \theta^2+\int_\T|{\partial_x} \theta|^2-{M'}\int_\T \theta^2\leq M''\left(\int_\T f^2(t,\cdot)+\int_\T g^2(t,\cdot)\right),$$ so that, for a.e. $t$, we have
$$e^{-Mt}\int_\T \theta^2(t,\cdot)+\int_0^t e^{-M\tau}\int_\T |{\partial_x} \theta|^2\leq M''\int_0^t e^{-M\tau}\left(\int_\T f^2(\tau,\cdot)+\int_\T g^2(\tau,\cdot)\right)d\tau.$$ The conclusion follows by the using the inequality  $e^{-MT}\leq e^{-Mt}\leq 1.$
\end{proof}

\end{document}